\documentclass{birkau}
\usepackage{amsmath,amssymb,url}
\usepackage{hyperref}

\theoremstyle{plain}
\newtheorem{theorem}{Theorem}[section]

\newtheorem{proposition}[theorem]{Proposition}
\newtheorem{lemma}[theorem]{Lemma}
\newtheorem{corollary}[theorem]{Corollary}

\theoremstyle{definition}
\newtheorem{example}[theorem]{Example}
\newtheorem{definition}[theorem]{Definition}

\theoremstyle{remark}
\newtheorem{remark}[theorem]{Remark}

\newcommand{\PosId}{\operatorname{Id}^{+}}
\newcommand{\Sk}{\operatorname{Sk}}
\newcommand{\tprod}[2]{#1\ast_{\tau}#2}
\newcommand{\CCtau}{\operatorname{CC}_{\tau}}
\newcommand{\CCm}{\operatorname{CC}_{\mathrm m}}
\newcommand{\layer}[1]{L_{#1}}
\newcommand{\Ctau}{\mathcal C_{\tau}}
\newcommand{\Cm}{\mathcal C_{\mathrm m}}
\newcommand{\Ltau}{\mathcal L_{\tau}}
\newcommand{\Lm}{\mathcal L_{\mathrm m}}

\makeatletter
\newcommand{\taggeditem}[2]{%
 \item[#1]%
 \def\@currentlabel{#1}%
 \label{#2}%
}
\makeatother

\begin{document}

\title[Resolved-transport decomposition]{\raggedright Canonical resolved-transport decomposition and reconstruction of local-unit-aligned ordered semigroups}

\author[S. Jenei]{S\'andor Jenei}
\address{Institute of Mathematics and Informatics\\
Eszterh\'azy K\'aroly Catholic University\\
Hungary\\
Institute of Mathematics and Informatics\\
University of P\'ecs\\
Hungary}
\urladdr{https://jenei.ttk.pte.hu/home.html}
\email{jenei.sandor@uni-eszterhazy.hu, jenei@ttk.pte.hu}


\begin{abstract}
This paper develops a canonical decomposition and reconstruction theory for local-unit-aligned ordered semigroups. Two coherentizations of the positive-idempotent skeleton are introduced. The finer one records the closure forced within local-unit blocks, while the multiplication-coherent quotient yields a join-semilattice of canonical blocks. Components are the fibers of these blocks.

For comparable blocks $A\le B$, each positive idempotent $q\in B$ defines a transport homomorphism $x\mapsto xq$ from the component over $A$ to that over $B$. Keeping all such maps gives a resolved-transport family, which replaces the single connecting map used in an ordinary direct system. These canonical data reconstruct multiplication without additional assumptions: two elements are transported to their join component, and their product is the least product of corresponding transported images. They also determine every comparison directed from a lower component to a higher one.

The full ambient order is recovered under any of three explicit order-recovery conditions. In particular, every totally ordered local-unit-aligned semigroup is completely reconstructed by its resolved-transport data. A complementary result recovers the order from componentwise order duality when a suitable component-preserving anti-automorphism is available.

When every block receiving a proper transition is lower-pointed, its least positive idempotent selects the pointwise least map in each resolved-transport family. If these receiving least idempotents are also skeleton-monotone, the selected maps satisfy the ordinary direct-system identities and recover multiplication by the usual single-map rule. They also define the associated directed-lexicographic relation, whose agreement with the ambient order is a separate order-recovery condition. Examples show both why the additional hypotheses are needed and why resolved transports cannot in general be replaced by ordinary transition maps.
\end{abstract}

\maketitle

\section{Introduction}

Decomposition methods in semigroup theory usually start from the same problem: one wants to replace a global multiplication by local pieces and by data that govern their interaction.
Classical examples include Green's relations and principal factors, Rees quotients and Rees matrix representations, and semilattice decompositions of regular or inverse semigroups; see, for example, \cite{Green1951,Rees1940,Clifford1941,CliffordPreston1961,Howie1995}.
In these theories the indexing object may come from ideals, from $\mathcal J$-classes, from idempotents, or from a semilattice of components.
The reconstruction data may then be sandwich matrices, connecting homomorphisms, or the operation of the indexing semilattice.
P\l{}onka's construction gives a universal-algebraic version of the same pattern: an algebra is assembled from components over a semilattice, and operations are evaluated after the arguments have been moved to the component selected by the semilattice operation \cite{Plonka1967}.

For ordered semigroups this problem has an additional order-theoretic side.
It is not enough to recover the multiplication; one also has to recover the ambient order.
Classical work on ordered semigroups includes Archimedean decompositions, ideal extensions, and ordinal-sum phenomena in naturally ordered or totally ordered commutative settings \cite{Clifford1954,Clifford1958,Saito1968,Saito1976I,Saito1976II,Saito1976III,Satyanarayana1979Arch,Gabovich1976,Hulin1976,KehayopuluTsingelis2003,KehayopuluTsingelis2006}.
In the finite commutative chain case, related structures have also been studied as tomonoids, through structural questions on totally ordered monoids, and through Rees coextensions and constructive representation procedures \cite{EvansEtAl2001,Whipple2005,Horcik2010,Vetterlein2015,Vetterlein2016Pos,Vetterlein2017Rep,PetrikVetterlein2014,PetrikVetterlein2016,PetrikVetterlein2017,PetrikVetterlein2019,Vetterlein2016Real}.
The present paper belongs to this decomposition tradition, but the invariant used here is different.
Instead of decomposing by ideals, Green classes, Archimedean equivalence, or externally prescribed construction data, we decompose by the behavior of local units.

The local-unit condition used below is deliberately intrinsic.
For an element $x$, consider the right and left local-unit sets
\[
\{z:xz=x\},
\qquad
\{z:zx=x\}.
\]
We say that the semigroup is local-unit-complete if these sets have greatest elements for every $x$.
It is local-unit-aligned if the two greatest local units coincide and their common value is positive.
This common value is denoted by $\tau(x)$.
Thus the map $\tau$ assigns to each element its canonical two-sided local unit, and its values lie in the ordered set $\PosId(\mathbf S)$ of positive idempotents.
This positive-idempotent skeleton is the raw index set from which the decompositions in this paper are extracted.

A first attempt would be to decompose $S$ into the exact $\tau$-layers $\tau^{-1}(p)$, one for each positive idempotent $p$.
This idempotent-by-idempotent decomposition is usually too fine.
Even in the finite totally ordered monoid setting, the useful canonical decomposition is obtained by merging positive idempotents into blocks until multiplication has a well-defined target block \cite{JeneiLUARepresentation}.
The central coarsening is the finest $\tau$-multiplication-coherent partition.
Its blocks form the coarse skeleton on which the later reconstruction is based.
If $B$ is such a block, the associated component is
\[
\layer{B}:=\{x\in S:\tau(x)\in B\}.
\]

In the finite chain situation this coarse skeleton is again a finite chain.
There the components can be organized as a direct system: for each lower component and higher component one has a transition homomorphism moving elements upward, and products are computed after the relevant factors have been moved to a common component.
The finite-chain rigidity theorem says more: the proper transition maps in that direct system are forced to be unit-constant \cite{JeneiLUARepresentation}.
This explains the finite picture, but it also points to the obstruction addressed here.
Without finiteness and totality, the positive-idempotent skeleton need not be a finite chain, and one distinguished transition map from a lower component to a higher one no longer carries enough information.

The present paper develops the corresponding general reconstruction theory for local-unit-aligned ordered semigroups.
There are two canonical coarsenings of the positive-idempotent skeleton.
The first is the finest $\tau$-stable partition.
It records the first closure forced by products taken inside a tentative local-unit block.
The second, and the one used for reconstruction, is the finest $\tau$-multiplication-coherent partition, denoted by $\Cm(\mathbf S)$.
It is obtained by making exactly the additional identifications needed so that the block of $\tau(xy)$ is determined by the blocks of $\tau(x)$ and $\tau(y)$.
Equivalently, the blocks of $\Cm(\mathbf S)$ form a quotient skeleton with a natural join operation.
This quotient skeleton plays, in the present ordered-semigroup setting, the role played by the index semilattice in strong semilattice and P\l{}onka-type constructions.

The replacement for a single transition map is a resolved family of transports.
Let $A,B\in\Cm(\mathbf S)$, and suppose that $A\le_{\rm m}B$ in the quotient skeleton.
For an element $x\in\layer{A}$, there is generally no canonical single image of $x$ in $\layer{B}$.
Instead, every positive idempotent $q$ in the receiving block $B$ gives an image
\[
\lambda^{A,B}_{q}(x)=xq
\qquad(q\in B).
\]
The family $(\lambda^{A,B}_{q})_{q\in B}$ is the resolved transport from $\layer{A}$ to $\layer{B}$.
It is not extra structure imposed on the semigroup; it is obtained from the original multiplication by keeping all target-idempotent actions visible.
Lower-pointedness of the receiving blocks selects a pointwise least map from each resolved-transport family. When the corresponding least positive idempotents are skeleton-monotone on the receiving part of the skeleton, these selected maps satisfy the ordinary direct-system identities.

The first reconstruction theorem is multiplicative.
For $x\in\layer{A}$ and $y\in\layer{B}$, the product $xy$ lies in the join component $\layer{A\vee B}$.
It is recovered by transporting both arguments to this join component through the same target idempotent and then taking the least of the resulting matching products:
\[
xy=
\min\{\lambda^{A,A\vee B}_{q}(x)\lambda^{B,A\vee B}_{q}(y):q\in A\vee B\}.
\]
Thus the formula has the familiar shape of a semilattice-indexed construction: the factors are first moved to the component selected by the index operation, and the product is then read there.
The new point is that the movement is resolved rather than single-valued.
Under the lower-pointed and skeleton-monotonicity hypotheses isolated below, the displayed formula collapses to the usual direct-system product rule.

The order requires a separate discussion.
The resolved-transport data always recover comparisons from a lower block to a higher block.
If $A\le_{\rm m}B$, $x\in\layer{A}$, and $y\in\layer{B}$, then
\[
 x\le y
 \quad\Longleftrightarrow\quad
 \lambda^{A,B}_{q}(x)\le y\text{ for some }q\in B.
\]
This is an intrinsic lower-to-higher test in the resolved system.
However, it does not by itself decide all cross-component comparisons.
In particular, a reverse comparison from a higher block to a lower block, or a comparison between incomparable blocks, requires additional order information.

The paper isolates three explicit order-recovery packages under which the resolved-transport system reconstructs the full ambient order.
They are skeleton orientation, skeleton-supported cross-block totality, and skeleton-supported strict upper-shadow reflection.
These are hypotheses on the cross-component order relative to the multiplication-derived skeleton and the transport maps.
They are not formal consequences of the resolved-transport data alone: a two-component example exhibits identical canonical quotient skeletons,
component ordered semigroups, and resolved transports, but different reverse cross-component comparisons.
Under any one of the three packages, the canonical resolved-transport system reconstructs the full ordered semigroup.
In particular, the cross-block-totality package applies to totally ordered semigroups.
The main text also identifies the implications among these order principles in the lower-pointed, skeleton-monotone collapse setting and isolates the extra endpoint rigidity supplied by totality together with skeleton orientation.

There is also a separate order-duality mechanism, still formulated at the ordered-semigroup level.
Suppose that the semigroup is skeleton-supported and carries a component-preserving order anti-automorphism $\nu$.
For $A<_{\rm m}B$, $x\in\layer{A}$, and $y\in\layer{B}$, the unresolved reverse comparison is transformed into
\[
y\le x
\quad\Longleftrightarrow\quad
\nu(x)\le\nu(y).
\]
Since $\nu$ preserves components, the comparison on the right is again a lower-to-higher comparison and is therefore decided by the resolved transports.
Proposition~\ref{prop:antiautomorphism-full-order-recovery} shows that the resolved-transport data, together with the componentwise action of $\nu$, reconstruct the full expanded ordered semigroup.
If the canonical quotient skeleton is a chain, skeleton-supportedness is automatic.

Balanced involutive residuated expansions are recorded in Remark~\ref{rem:involutive-order-recovery} as one specialization of this semigroup-level mechanism; the recovery argument itself uses only the ordered-semigroup structure and componentwise order duality.

The main contribution is therefore an internal decomposition and reconstruction theorem for the canonical data extracted from a given local-unit-aligned ordered semigroup.
Every such semigroup has the decomposition $\Lm(\mathbf S)$ induced by $\Cm(\mathbf S)$ and the canonical resolved-transport system $\mathfrak R(\mathbf S)$.
This system reconstructs multiplication and all lower-to-higher order comparisons without further assumptions.
It reconstructs the full ordered semigroup under any of the three explicit order-recovery packages described above.
A further full-order recovery mechanism is obtained by adjoining the componentwise action of a component-preserving order anti-automorphism.
As a consequence, $\mathfrak R(\mathbf S)$ recovers every local-unit-aligned totally ordered semigroup.
The finite-chain picture is retained in two respects: there are components, and there is transition-based reconstruction.
What changes is the shape of the canonical quotient skeleton and the form of
the connecting data: the canonical $\tau$-multiplication-coherent quotient
of the positive-idempotent skeleton, which is a finite chain in the finite totally ordered case, becomes a join-semilattice in general, while rigid single transition maps are replaced by resolved families of transports.

The present paper does not axiomatize resolved-transport systems independently of an underlying ordered semigroup and does not prove a converse realization theorem for arbitrary such systems.
Its reconstruction statements concern the canonical system $\mathfrak R(\mathbf S)$ extracted from a given $\mathbf S$.
The systematic arbitrary-chain separation family is deferred to the appendix, while the short examples and implication results needed to interpret the principal order-recovery hypotheses are placed at their point of use in the main body.

The paper is organized as follows.
Section~\ref{sec:preliminaries} recalls the local-unit-aligned setting and the elementary facts about positive idempotents.
Sections~\ref{sec:tau-stable} and~\ref{sec:tau-multiplication-coherent} construct the $\tau$-stable and $\tau$-multiplication-coherent partitions.
Section~\ref{sec:component-semigroups} proves that the induced components are again local-unit-aligned ordered semigroups.
Section~\ref{sec:quotient-skeleton-stage} develops quotient skeletons and resolved transports for arbitrary coherent partitions and proves product and lower-to-higher order recovery.
Section~\ref{sec:canonical-resolved-transport} specializes these constructions to $\Cm(\mathbf S)$, states the three explicit order-recovery packages, establishes their principal boundary phenomena, develops the semigroup-level recovery mechanism supplied by a component-preserving order anti-automorphism, and records the balanced involutive residuated specialization.
Section~\ref{sec:least-element-collapse} separates the least-target selection supplied by lower-pointedness from the skeleton-monotonicity needed for direct-system coherence, and determines the relations among the associated order-recovery principles.
Section~\ref{sect:canonical-resolved-decomposition} states the canonical reconstruction theorem and its principal order-recovery corollaries.
Appendix~\ref{sec:dependence-order-recovery-principles} gives separation models over arbitrary nontrivial chain skeletons.

\section{Preliminaries}
\label{sec:preliminaries}

By a \emph{chain} we mean a nonempty totally ordered set.
Throughout the paper, an \emph{ordered semigroup} means a partially ordered semigroup whose multiplication is isotone in both coordinates.
Thus
\[
\mathbf S=\langle S,\le,\cdot\rangle
\]
denotes such an ordered semigroup: $\langle S,\le\rangle$ is a partially ordered set, $\langle S,\cdot\rangle$ is a semigroup, and multiplication is isotone in both arguments,
\[
 x\le y \Longrightarrow xz\le yz\quad\text{and}\quad zx\le zy
 \qquad(z\in S).
\]
A \emph{totally ordered semigroup} is an ordered semigroup whose order is a chain.
No finiteness assumption and no global identity element are assumed.
Whenever totality, a maximum, or a minimum is used below, its existence is part of the relevant hypothesis.

\begin{definition}\label{def:local-unit-aligned}
Let $\mathbf S$ be an ordered semigroup.

\begin{enumerate}
\item An element $p\in S$ is \emph{idempotent} if $p^2=p$.

\item An element $p\in S$ is called \emph{positive} if it is order-inflationary on both sides, that is,
\[
 x\le xp\quad\text{and}\quad x\le px\qquad(x\in S).
\]
A positive idempotent is an element that is both positive and idempotent.
We write $\PosId(\mathbf S)$ for the set of positive idempotents of $\mathbf S$.
If $\mathbf S$ has a two-sided identity element $e$, positivity is equivalent to the usual condition $p\ge e$.

\item The ordered semigroup $\mathbf S$ is called \emph{local-unit-complete} if, for every $x\in S$, the two local-unit sets
\[
\{z\in S:xz=x\},
\qquad
\{z\in S:zx=x\}
\]
have greatest elements.
For a local-unit-complete ordered semigroup, define
\[
\tau_r(x):=\max\{z\in S:xz=x\},
\qquad
\tau_{\ell}(x):=\max\{z\in S:zx=x\}.
\]

\item A local-unit-complete ordered semigroup $\mathbf S$ is called \emph{local-unit-aligned} if for every $x\in S$ the two greatest local units coincide, $\tau_r(x)=\tau_{\ell}(x)$, and their common value is positive.
In this case we write $\tau(x):=\tau_r(x)=\tau_{\ell}(x)$ and call $\tau$ the \emph{local-unit map} of $\mathbf S$.
\end{enumerate}
\end{definition}

\begin{remark}[Relation with cut formulations and residuation]
\label{rem:relation-to-finite-monoid}
For an ordered monoid $\langle S,\le,\cdot,e\rangle$, positivity in the present sense is equivalent to being at least $e$.
For $x\in S$, consider the one-sided inequality cuts
\[
C_r(x):=\{z\in S:xz\le x\},
\qquad
C_{\ell}(x):=\{z\in S:zx\le x\}.
\]
Both cuts contain $e$.

Suppose that $C_r(x)$ has a greatest element $c_r(x)$.
Since $e\le c_r(x)$, the element $c_r(x)$ is positive, and hence
\[
x\le xc_r(x).
\]
The defining property of the cut gives the reverse inequality, so
\[
xc_r(x)=x.
\]
Moreover, every right local unit of $x$ belongs to $C_r(x)$.
Hence $c_r(x)$ is the greatest right local unit of $x$.
The analogous argument shows that, whenever $C_{\ell}(x)$ has a greatest element $c_{\ell}(x)$, this element is the greatest left local unit of $x$.

Consequently, whenever the relevant inequality cuts have greatest elements, the cut and equality-local-unit formulations coincide:
\[
\tau_r(x)=\max C_r(x),
\qquad
\tau_{\ell}(x)=\max C_{\ell}(x).
\]
Existence of the greatest equality local units alone does not, in an arbitrary partially ordered monoid, imply that the larger inequality cuts have greatest elements.

In a finite totally ordered monoid, however, both inequality cuts are nonempty finite chains and therefore have greatest elements.
Thus the equality-local-unit formulation adopted here agrees on that class with the cut formulation used in the finite-chain representation theorem of \cite{JeneiLUARepresentation}.

There is a second setting in which the cut maxima are supplied intrinsically.
If $\mathbf S$ is a residuated po-semigroup, then residuation gives
\[
x\backslash x=\max C_r(x),
\qquad
x/x=\max C_{\ell}(x).
\]
Suppose, in addition, that $\mathbf S$ is balanced, in the sense that
\[
x\backslash x=x/x
\]
and this common element is positive for every $x$.
Writing
\[
p:=x\backslash x=x/x,
\]
residuation gives $xp\le x$ and $px\le x$, while positivity gives the reverse inequalities.
Hence
\[
xp=x=px.
\]
Since $p$ is the greatest element of both inequality cuts, it also dominates every right or left local unit of $x$.
The ordered-semigroup reduct is therefore local-unit-aligned, with
\[
\tau(x)=x\backslash x=x/x.
\]

Thus the cut formulation from the finite totally ordered setting and the self-residual formulation from the balanced residuated setting both specialize to the equality-local-unit formulation used throughout the present paper.
For general ordered semigroups, the equality formulation is preferable because there need be no ambient identity ensuring positivity and no residuation ensuring the existence of the cut maxima.
\end{remark}

Unless explicitly stated otherwise, all $\tau$-based notation below is used only for local-unit-aligned ordered semigroups and is formed with respect to the local-unit map of the semigroup under discussion.

\begin{definition}
The ordered set
\[
 \Sk(\mathbf S):=\langle \PosId(\mathbf S),\le\!\upharpoonright_{\PosId(\mathbf S)}\rangle
\]
is called the \emph{skeleton} of $\mathbf S$.
We say that $\mathbf S$ is \emph{positive-idempotent-linear} if $\Sk(\mathbf S)$ is a chain.
Every totally ordered semigroup is positive-idempotent-linear, but the converse need not hold.
\end{definition}

\begin{lemma}\label{lem:basic_tau}
Let $\mathbf S$ be local-unit-aligned and let $x\in S$.
Then:
\begin{enumerate}
\item\label{item:positive} $\tau(x)$ is positive;
\item\label{item:local-unit-equalities} $x\tau(x)=x=\tau(x)x$;
\item\label{item:tau-idempotent} $\tau(x)$ is idempotent, and hence $\tau(x)\in\PosId(\mathbf S)$;
\item\label{item:tau-fixes-idempotents} if $u\in\PosId(\mathbf S)$, then $\tau(u)=u$;
\item\label{item:below-tau-gives-unit} if $u\in\PosId(\mathbf S)$ and $u\le \tau(x)$, then $xu=x=ux$.
\end{enumerate}
\end{lemma}

\begin{proof}
Since $\mathbf S$ is local-unit-aligned, the greatest right and left local units of $x$ coincide, and their common value $\tau(x)$ is positive by definition.
This proves item~\eqref{item:positive}.
By the defining local-unit equalities, $x\tau(x)=x=\tau(x)x$, proving item~\eqref{item:local-unit-equalities}.

For idempotence, these equalities give $x\tau(x)^2=(x\tau(x))\tau(x)=x\tau(x)=x$, and similarly $\tau(x)^2x=\tau(x)(\tau(x)x)=\tau(x)x=x$.
Thus $\tau(x)^2$ is both a right and a left local unit of $x$, and maximality of $\tau(x)$ gives $\tau(x)^2\le\tau(x)$.
Applying positivity of $\tau(x)$ to the element $\tau(x)$ itself gives $\tau(x)\le\tau(x)^2$.
Hence $\tau(x)^2=\tau(x)$, proving item~\eqref{item:tau-idempotent}.

If $u\in\PosId(\mathbf S)$, then $u^2=u$, so $u$ is both a right and a left local unit of itself, hence $u\le\tau(u)$.
Since $\tau(u)$ is a left local unit of $u$, we have $\tau(u)u=u$.
Since $u$ is positive, $\tau(u)\le \tau(u)u=u$.
Thus $\tau(u)\le u$, and therefore $\tau(u)=u$.

Finally, if $u\le\tau(x)$, then isotonicity gives $xu\le x\tau(x)=x$ and $ux\le\tau(x)x=x$, while positivity of $u$ gives $x\le xu$ and $x\le ux$.
Thus $xu=x=ux$.
\end{proof}

\begin{proposition}[Centrality of positive idempotents]\label{prop:positive-idempotents-central}
Let $\mathbf S$ be local-unit-aligned.
Then every positive idempotent of $\mathbf S$ is central.
Equivalently, for every $p\in\PosId(\mathbf S)$ and every $x\in S$, one has
\[
px=xp.
\]
\end{proposition}

\begin{proof}
Fix $p\in\PosId(\mathbf S)$ and $x\in S$.
Since $p(px)=px$, the element $p$ is a left local unit of $px$, hence $p\le\tau(px)$.
Lemma~\ref{lem:basic_tau}\eqref{item:below-tau-gives-unit} applied to the element $px$ gives $(px)p=px$.
Similarly, since $(xp)p=xp$, the element $p$ is a right local unit of $xp$, hence $p\le\tau(xp)$.
Applying Lemma~\ref{lem:basic_tau}\eqref{item:below-tau-gives-unit} to the element $xp$ gives $p(xp)=xp$.
Both displayed products are equal to $pxp$, and therefore $px=xp$.
\end{proof}

\begin{lemma}\label{lem:idempotents-max}
Let $\mathbf S$ be local-unit-aligned, and let $p,q\in\PosId(\mathbf S)$.
Then $pq=qp\in\PosId(\mathbf S)$, and $pq$ is the join $p\vee q$ in the ordered set $\Sk(\mathbf S)$.
In particular, if $p\le q$, then $pq=qp=q$.
\end{lemma}

\begin{proof}
The equality $pq=qp$ follows from Proposition~\ref{prop:positive-idempotents-central}.
Hence $(pq)^2=pqpq=p^2q^2=pq$, so $pq$ is idempotent.
It is positive: for any $x\in S$, positivity of $q$, then positivity of $p$, gives $x\le xq\le (xq)p=xpq$; similarly $x\le qx\le p(qx)=pqx$.
Thus $pq\in\PosId(\mathbf S)$.
Moreover, positivity gives $p\le pq$ and $q\le pq$, so $pq$ is an upper bound of $p$ and $q$ inside $\Sk(\mathbf S)$.
If $r\in\PosId(\mathbf S)$ is another upper bound, then $p\le r$ and $q\le r$.
By isotonicity, $pq\le rq$, and the special comparable case applied to $q\le r$ gives $rq=r$: indeed positivity gives $r\le rq$, while isotonicity gives $rq\le rr=r$.
Thus $pq\le r$, proving that $pq=p\vee q$.
The final assertion is the join identity when $p\le q$.
\end{proof}

\begin{proposition}\label{prop:tau-of-product}
Let $\mathbf S$ be local-unit-aligned.
Then for all $x,y\in S$,
\[
\tau(x)\le \tau(xy),\qquad \tau(y)\le \tau(xy),
\]
and
\[
\tau(x)\tau(y)=\tau(x)\vee\tau(y)\le\tau(xy).
\]
\end{proposition}

\begin{proof}
Since $y\tau(y)=y$, we have $xy\tau(y)=xy$, so $\tau(y)$ is a right local unit of $xy$ and therefore $\tau(y)\le\tau(xy)$.
Likewise, since $\tau(x)x=x$, we have $\tau(x)xy=xy$, so $\tau(x)$ is a left local unit of $xy$ and therefore $\tau(x)\le\tau(xy)$.
By Lemma~\ref{lem:basic_tau}\eqref{item:tau-idempotent}, $\tau(xy)$ is a positive idempotent upper bound of $\tau(x)$ and $\tau(y)$.
By Lemma~\ref{lem:idempotents-max}, the product $\tau(x)\tau(y)$ is the join of $\tau(x)$ and $\tau(y)$ in $\Sk(\mathbf S)$, and hence it is below $\tau(xy)$.
\end{proof}

\begin{definition}
Let $\mathbf S$ be local-unit-aligned.
\begin{enumerate}
\item For $u\in\PosId(\mathbf S)$, the \emph{$u$-layer} is $\layer{u}:=\{x\in S:\tau(x)=u\}$.
\item For $A\subseteq\PosId(\mathbf S)$, put $\layer{A}:=\tau^{-1}(A)=\{x\in S:\tau(x)\in A\}$.
\item For $A,B\subseteq\PosId(\mathbf S)$, define their \emph{$\tau$-saturated product} by
\[
 \tprod{A}{B}:=\tau(\layer{A}\cdot\layer{B})
 =\{\tau(xy):x\in\layer{A},\ y\in\layer{B}\}.
\]
\end{enumerate}
\end{definition}

\begin{lemma}\label{lem:tprod-monotone}
If $A\subseteq A'$ and $B\subseteq B'$, then $\tprod{A}{B}\subseteq \tprod{A'}{B'}$.
In particular, $A\subseteq A'\Longrightarrow \tprod{A}{A}\subseteq \tprod{A'}{A'}$.
\end{lemma}

\begin{proof}
The inclusions $\layer{A}\subseteq\layer{A'}$ and $\layer{B}\subseteq\layer{B'}$ imply $\layer{A}\cdot\layer{B}\subseteq\layer{A'}\cdot\layer{B'}$, and applying $\tau$ gives the result.
\end{proof}

\begin{remark}\label{rem:A-subset-tprodAA}
For every $A\subseteq\PosId(\mathbf S)$, $A\subseteq\tprod{A}{A}$.
Indeed, if $u\in A$, then $u\in\layer{A}$ by Lemma~\ref{lem:basic_tau}\eqref{item:tau-fixes-idempotents}, and $u=\tau(u^2)\in\tprod{A}{A}$.
\end{remark}

\section{\texorpdfstring{$\tau$-stable partitions and the $\tau$-cohesive decomposition}{tau-stable partitions and the tau-cohesive decomposition}}
\label{sec:tau-stable}

We use the standard refinement order on partitions: a partition $\mathcal P$ of a set $X$ \emph{refines} a partition $\mathcal Q$ of $X$ if every block of $\mathcal P$ is contained in a block of $\mathcal Q$.

The role of this section is preparatory.
It records the first closure forced on local-unit blocks before the stronger multiplication-coherence condition is imposed.
As will follow from Theorem~\ref{thm:tau-refines-mult}, $\tau$-multiplication-coherent blocks are unions of $\tau$-stable blocks.
Thus the decomposition induced by the finest $\tau$-stable partition, later denoted by $\Ltau(\mathbf S)$ and called the $\tau$-cohesive decomposition, is a first structural approximation to the later multiplicative coherentization.
It also provides a natural starting point for a structural analysis of $\tau$-multiplication-cohesive ordered semigroups.

\begin{definition}
Let $\mathbf S$ be local-unit-aligned.

\begin{enumerate}
\item A nonempty set $A\subseteq\PosId(\mathbf S)$ is called \emph{$\tau$-stable} if $\tprod{A}{A}=A$.
\item A partition $\mathcal P$ of $\PosId(\mathbf S)$ is called \emph{$\tau$-stable} if every block of $\mathcal P$ is $\tau$-stable.
\end{enumerate}
\end{definition}

Thus a $\tau$-stable block is meant to isolate a collection of local units whose associated layers are closed under taking the local unit of an internal product.

\begin{definition}
Let $\mathcal P$ be a partition of $\PosId(\mathbf S)$.
Define a graph on the set of blocks of $\mathcal P$ by joining two blocks $A$ and $B$ whenever $\tprod{A}{A}\cap \tprod{B}{B}\neq\varnothing$.
We denote by $\CCtau(\mathcal P)$ the partition obtained by merging precisely those blocks that lie in the same connected component of this graph.
\end{definition}

\begin{lemma}\label{lem:CCtau-monotone}
If a partition $\mathcal P$ refines a partition $\mathcal Q$, then $\CCtau(\mathcal P)$ refines $\CCtau(\mathcal Q)$.
\end{lemma}

\begin{proof}
Let $A,B$ be blocks of $\mathcal P$ with $\tprod{A}{A}\cap \tprod{B}{B}\neq\varnothing$.
Let $A',B'$ be the blocks of $\mathcal Q$ containing $A$ and $B$, respectively.
By Lemma~\ref{lem:tprod-monotone}, $\tprod{A}{A}\subseteq \tprod{A'}{A'}$ and $\tprod{B}{B}\subseteq \tprod{B'}{B'}$.
Hence $\tprod{A'}{A'}\cap \tprod{B'}{B'}\neq\varnothing$.
If $A'=B'$, then the two image blocks already coincide.
If $A'\neq B'$, the displayed nonempty intersection shows that $A'$ and $B'$ are joined by an edge in the graph of $\mathcal Q$.
Thus every edge in the graph of $\mathcal P$ maps either to a single vertex or to an edge in the graph of $\mathcal Q$.
Therefore connected components for $\mathcal P$ are contained in connected components for $\mathcal Q$, which exactly says that $\CCtau(\mathcal P)$ refines $\CCtau(\mathcal Q)$.
\end{proof}

\begin{proposition}\label{prop:CCtau-fixed-points}
For a partition $\mathcal P$ of $\PosId(\mathbf S)$, the following are equivalent:
\begin{enumerate}
\item\label{item:CCtau-fixed} $\CCtau(\mathcal P)=\mathcal P$;
\item\label{item:CCtau-stable} $\mathcal P$ is $\tau$-stable.
\end{enumerate}
\end{proposition}

\begin{proof}
Assume first that $\CCtau(\mathcal P)=\mathcal P$ and let $A\in\mathcal P$.
Take $u\in \tprod{A}{A}$ and let $B\in\mathcal P$ be the unique block containing $u$.
By Remark~\ref{rem:A-subset-tprodAA}, $B\subseteq \tprod{B}{B}$.
Since $u\in \tprod{A}{A}\cap B$, it follows that $\tprod{A}{A}\cap \tprod{B}{B}\neq\varnothing$.
Hence $A$ and $B$ are joined by an edge, so they would be merged by $\CCtau$.
Because $\CCtau(\mathcal P)=\mathcal P$, we must have $A=B$.
Therefore every element of $\tprod{A}{A}$ lies in $A$, that is, $\tprod{A}{A}\subseteq A$.
Together with Remark~\ref{rem:A-subset-tprodAA}, this gives $\tprod{A}{A}=A$.
Thus $\mathcal P$ is $\tau$-stable.

Conversely, assume that $\mathcal P$ is $\tau$-stable.
Then for distinct blocks $A,B\in\mathcal P$, $\tprod{A}{A}=A$, $\tprod{B}{B}=B$, so $\tprod{A}{A}\cap \tprod{B}{B}=A\cap B=\varnothing$.
Hence the graph defining $\CCtau(\mathcal P)$ has no edges between distinct blocks, so $\CCtau(\mathcal P)=\mathcal P$.
\end{proof}

\begin{lemma}\label{lem:omega-union-equivalence}
Let $(E_n)_{n\ge 0}$ be an increasing sequence of equivalence relations on a set $X$, meaning $E_n\subseteq E_{n+1}$ for all $n$.
Define $E_\omega:=\bigcup_{n\ge 0}E_n$.
Then $E_\omega$ is an equivalence relation.
Moreover, if $F$ is an equivalence relation on $X$ and $E_n\subseteq F$ for every $n$, then $E_\omega\subseteq F$.
\end{lemma}

\begin{proof}
Reflexivity and symmetry are immediate.
For transitivity, suppose $xE_\omega y$ and $yE_\omega z$.
Then $xE_m y$ and $yE_n z$ for some $m,n$.
With $k:=\max\{m,n\}$, monotonicity gives $xE_k y$ and $yE_k z$, hence $xE_k z$.
Thus $xE_\omega z$.
The final assertion follows directly from $E_\omega=\bigcup_n E_n$.
\end{proof}

\begin{definition}\label{def:Ctau-omega}
Let $\mathbf S$ be local-unit-aligned.
Start from the singleton partition $\mathcal P_{0}:=\bigl\{\{u\}:u\in\PosId(\mathbf S)\bigr\}$.
Define recursively $\mathcal P_{n+1}:=\CCtau(\mathcal P_{n})$ for $n\ge 0$, and let $\approx_n$ be the equivalence relation on $\PosId(\mathbf S)$ whose classes are the blocks of $\mathcal P_n$.
Since every step only merges blocks, the relations $\approx_n$ form an increasing sequence.
Set
\[
 u\approx_{\omega}^{\tau} v
 \quad\Longleftrightarrow\quad
 u\approx_n v\ \text{for some }n\ge 0.
\]
By Lemma~\ref{lem:omega-union-equivalence}, $\approx_{\omega}^{\tau}$ is an equivalence relation.
Define
\[
\Ctau(\mathbf S):=\PosId(\mathbf S)/{\approx_{\omega}^{\tau}}.
\]
We call $\Ctau(\mathbf S)$ the \emph{$\tau$-cohesive partition} of $\mathbf S$.
The induced family $\Ltau(\mathbf S):=\{\layer{A}:A\in\Ctau(\mathbf S)\}$ is the \emph{$\tau$-cohesive decomposition} of $\mathbf S$.
If the iteration stabilizes at a finite stage, then this omega-limit is the stable partition at that stage.
\end{definition}

\begin{theorem}\label{thm:finest-tau-stable}
The partition $\Ctau(\mathbf S)$ is the finest $\tau$-stable partition of $\PosId(\mathbf S)$.
Equivalently, $\Ltau(\mathbf S)$ is the finest decomposition of $S$ induced by a $\tau$-stable partition.
\end{theorem}

\begin{proof}
We first prove that $\Ctau(\mathbf S)$ is $\tau$-stable.
Let $A\in\Ctau(\mathbf S)$ and let $u\in\tprod{A}{A}$.
Choose $x,y\in\layer{A}$ such that $u=\tau(xy)$.
Put $a:=\tau(x)$ and $b:=\tau(y)$.
Since $a,b\in A$, by the definition of the omega-limit relation there is $n\ge 0$ such that $a\approx_n b$; let $P\in\mathcal P_n$ be the block containing them.
Then $u\in\tprod{P}{P}$.
Let $Q\in\mathcal P_n$ be the block containing $u$.
By Remark~\ref{rem:A-subset-tprodAA}, $Q\subseteq\tprod{Q}{Q}$, so $\tprod{P}{P}\cap\tprod{Q}{Q}\ne\varnothing$.
Thus $P$ and $Q$ are merged by $\CCtau$, and hence $a\approx_{n+1}u$.
Therefore $u\in A$.
This proves $\tprod{A}{A}\subseteq A$, while the reverse inclusion is Remark~\ref{rem:A-subset-tprodAA}.
Thus $A$ is $\tau$-stable.

Let $\mathcal Q$ be any $\tau$-stable partition of $\PosId(\mathbf S)$.
Then $\CCtau(\mathcal Q)=\mathcal Q$ by Proposition~\ref{prop:CCtau-fixed-points}.
The singleton partition $\mathcal P_0$ refines $\mathcal Q$, so Lemma~\ref{lem:CCtau-monotone} yields inductively that every $\mathcal P_n$ refines $\mathcal Q$.
Hence each omega-limit class of $\Ctau(\mathbf S)$ is contained in a block of $\mathcal Q$.
Thus $\Ctau(\mathbf S)$ refines every $\tau$-stable partition, so it is the finest one.
The statement about induced decompositions follows by taking $\tau^{-1}$ of the blocks.
\end{proof}

\begin{definition}\label{def:tau-cohesive}
A local-unit-aligned ordered semigroup $\mathbf S$ is called \emph{$\tau$-cohesive} if its $\tau$-cohesive partition is trivial, that is, $\Ctau(\mathbf S)=\{\PosId(\mathbf S)\}$.
Equivalently, its $\tau$-cohesive decomposition is trivial, that is, $\Ltau(\mathbf S)=\{S\}$.
\end{definition}

\begin{remark}\label{rem:tau-cohesive-naming-forward}
The terminology introduced here is intentionally two-level.
For every local-unit-aligned ordered semigroup $\mathbf S$, the notation $\Ctau(\mathbf S)$ denotes a canonical partition of $\PosId(\mathbf S)$, called the \emph{$\tau$-cohesive partition}, whether or not it is trivial.
By contrast, the semigroup $\mathbf S$ itself is called \emph{$\tau$-cohesive} only when this canonical partition has a single block.
The reason for this naming convention becomes conceptually clear in Theorem~\ref{thm:Ctau-characterized-by-components}.
That theorem shows that $\Ctau(\mathbf S)$ is the unique $\tau$-stable partition of $\PosId(\mathbf S)$ whose associated components are all intrinsically $\tau$-cohesive.
By Proposition~\ref{prop:components-inherit-local-unit-alignment}, these components are also intrinsically local-unit-aligned, with intrinsic local-unit map given by the restriction of the ambient map.
\end{remark}

\section{\texorpdfstring{$\tau$-multiplication-coherent partitions and the $\tau$-multiplication-coherent decomposition}{tau-multiplication-coherent partitions and the tau-multiplication-coherent decomposition}}\label{sec:tau-multiplication-coherent}

\begin{definition}[Output blocks]\label{def:m-output-blocks}
Let $\mathcal P$ be a partition of $\PosId(\mathbf S)$, and let $A,B\in\mathcal P$.
Define the set of \emph{output blocks} of the pair $(A,B)$ by
\[
\operatorname{Out}_{\mathcal P}(A,B)
 :=
 \{C\in\mathcal P:C\cap\tprod{A}{B}\neq\varnothing\}.
\]
Choose $p\in A$ and $q\in B$.
Since $p$ and $q$ are positive idempotents, $\tau(p)=p$ and $\tau(q)=q$, so $p\in\layer{A}$ and $q\in\layer{B}$.
Hence
\[
\tau(pq)\in\tprod{A}{B},
\]
and therefore both $\tprod{A}{B}$ and $\operatorname{Out}_{\mathcal P}(A,B)$ are nonempty.
The partition $\mathcal P$ is called \emph{$\tau$-multiplication-coherent} if $\operatorname{Out}_{\mathcal P}(A,B)$ is a singleton for all blocks $A,B\in\mathcal P$.
Equivalently, products of layers have a well-defined target component with respect to $\mathcal P$: the $\tau$-component of a product is determined by the two input blocks alone, rather than by the particular representatives chosen inside their layers.
\end{definition}

\begin{lemma}\label{lem:coherent-implies-stable}
Every $\tau$-multiplication-coherent partition is $\tau$-stable.
\end{lemma}

\begin{proof}
Let $\mathcal P$ be $\tau$-multiplication-coherent and let $A\in\mathcal P$.
By Remark~\ref{rem:A-subset-tprodAA}, the set $\tprod{A}{A}$ meets $A$, so $A\in\operatorname{Out}_{\mathcal P}(A,A)$.
By coherence, $\operatorname{Out}_{\mathcal P}(A,A)$ is a singleton.
Hence every block meeting $\tprod{A}{A}$ is $A$, and therefore $\tprod{A}{A}\subseteq A$.
Together with Remark~\ref{rem:A-subset-tprodAA}, this yields $\tprod{A}{A}=A$.
Hence $\mathcal P$ is $\tau$-stable.
\end{proof}

\begin{definition}
Let $\mathcal P$ be a partition of $\PosId(\mathbf S)$.
Define a graph on the set of blocks of $\mathcal P$ by joining two blocks $C$ and $D$ whenever there exist blocks $A,B\in\mathcal P$ such that
\[
C,D\in\operatorname{Out}_{\mathcal P}(A,B).
\]
Equivalently, the edge condition says that both $C$ and $D$ meet $\tprod{A}{B}$.
We denote by $\CCm(\mathcal P)$ the partition obtained by merging precisely those blocks that lie in the same connected component of this graph.
\end{definition}

\begin{lemma}\label{lem:CCx-monotone}
If a partition $\mathcal P$ refines a partition $\mathcal Q$, then $\CCm(\mathcal P)$ refines $\CCm(\mathcal Q)$.
\end{lemma}

\begin{proof}
Let $C,D\in\mathcal P$ be joined by an edge in the graph defining $\CCm(\mathcal P)$.
Then there exist $A,B\in\mathcal P$ such that $C,D\in\operatorname{Out}_{\mathcal P}(A,B)$.
Let $A',B',C',D'$ be the blocks of $\mathcal Q$ containing $A,B,C,D$, respectively.
By Lemma~\ref{lem:tprod-monotone}, $\tprod{A}{B}\subseteq \tprod{A'}{B'}$.
Hence $C',D'\in\operatorname{Out}_{\mathcal Q}(A',B')$.
If $C'=D'$, then the edge collapses to one block of $\mathcal Q$.
If $C'\neq D'$, then $C'$ and $D'$ are joined by an edge in the graph defining $\CCm(\mathcal Q)$.
Thus every edge for $\mathcal P$ maps either to a single block or to an edge for $\mathcal Q$, and therefore connected components for $\mathcal P$ map into connected components for $\mathcal Q$.
This proves that $\CCm(\mathcal P)$ refines $\CCm(\mathcal Q)$.
\end{proof}

\begin{proposition}\label{prop:CCx-fixed-points}
For a partition $\mathcal P$ of $\PosId(\mathbf S)$, the following are equivalent:
\begin{enumerate}
\item\label{item:CCx-fixed} $\CCm(\mathcal P)=\mathcal P$;
\item\label{item:CCx-coherent} $\mathcal P$ is $\tau$-multiplication-coherent.
\end{enumerate}
\end{proposition}

\begin{proof}
Assume first that $\CCm(\mathcal P)=\mathcal P$.
Let $A,B\in\mathcal P$.
If $C,D\in\operatorname{Out}_{\mathcal P}(A,B)$, then $C$ and $D$ are joined by an edge in the graph defining $\CCm(\mathcal P)$.
Since the partition is unchanged, $C=D$.
Thus $\operatorname{Out}_{\mathcal P}(A,B)$ is a singleton, proving coherence.

Conversely, assume that $\mathcal P$ is $\tau$-multiplication-coherent.
Then for every pair $A,B\in\mathcal P$, the set $\operatorname{Out}_{\mathcal P}(A,B)$ has only one element.
Therefore no edge in the graph defining $\CCm(\mathcal P)$ joins two distinct blocks.
Hence $\CCm(\mathcal P)=\mathcal P$.
\end{proof}

\begin{definition}\label{def:Cm-omega}
Let $\mathbf S$ be local-unit-aligned.
Start again from the singleton partition $\mathcal Q_{0}:=\bigl\{\{u\}:u\in\PosId(\mathbf S)\bigr\}$.
Define recursively $\mathcal Q_{n+1}:=\CCm(\mathcal Q_{n})$ for $n\ge 0$, and let $\equiv_n$ be the equivalence relation on $\PosId(\mathbf S)$ whose classes are the blocks of $\mathcal Q_n$.
Since every step only merges blocks, the relations $\equiv_n$ form an increasing sequence.
Set
\[
 u\equiv_{\omega}^{\mathrm m} v
 \quad\Longleftrightarrow\quad
 u\equiv_n v\ \text{for some }n\ge 0.
\]
By Lemma~\ref{lem:omega-union-equivalence}, $\equiv_{\omega}^{\mathrm m}$ is an equivalence relation.
Define
\[
\Cm(\mathbf S):=\PosId(\mathbf S)/{\equiv_{\omega}^{\mathrm m}}.
\]
We call $\Cm(\mathbf S)$ the \emph{canonical $\tau$-multiplication-coherent partition} of $\mathbf S$, or equivalently the \emph{multiplicative coherentization} of $\mathbf S$.
The induced family $\Lm(\mathbf S):=\{\layer{A}:A\in\Cm(\mathbf S)\}$ is the \emph{canonical $\tau$-multiplication-coherent decomposition} of $\mathbf S$, or equivalently the \emph{multiplicative coherentization decomposition}.
If the iteration stabilizes at a finite stage, then this omega-limit is the fixed point at that stage.
\end{definition}

\begin{definition}\label{def:tau-multiplication-cohesive}
A local-unit-aligned ordered semigroup $\mathbf S$ is called \emph{$\tau$-multiplication-cohesive} if its canonical $\tau$-multiplication-coherent partition is trivial, that is, $\Cm(\mathbf S)=\{\PosId(\mathbf S)\}$.
Equivalently, its canonical $\tau$-multiplication-coherent decomposition is trivial, that is, it has a single component: $\Lm(\mathbf S)=\{S\}$.
\end{definition}

The following theorem gives the universal property which justifies viewing $\Cm(\mathbf S)$ as the canonical multiplication-coherent quotient.

\begin{theorem}\label{thm:finest-coherent}
The partition $\Cm(\mathbf S)$ is the finest $\tau$-multiplication-coherent partition of $\PosId(\mathbf S)$.
\end{theorem}

\begin{proof}
We first prove that $\Cm(\mathbf S)$ is $\tau$-multiplication-coherent.
Let $A,B\in\Cm(\mathbf S)$ and let $u,v\in\tprod{A}{B}$.
Choose $x_1,x_2\in\layer{A}$ and $y_1,y_2\in\layer{B}$ such that $u=\tau(x_1y_1)$ and $v=\tau(x_2y_2)$.
Since $\tau(x_1),\tau(x_2)\in A$, there is $n_A\ge 0$ such that $\tau(x_1)\equiv_{n_A}\tau(x_2)$.
Since $\tau(y_1),\tau(y_2)\in B$, there is $n_B\ge 0$ such that $\tau(y_1)\equiv_{n_B}\tau(y_2)$.
After replacing both stages by $n:=\max\{n_A,n_B\}$, using monotonicity of the iteration, we have both $\tau(x_1)\equiv_n\tau(x_2)$ and $\tau(y_1)\equiv_n\tau(y_2)$.
Let $P,Q\in\mathcal Q_n$ be the blocks containing the two input local units.
Then $u,v\in\tprod{P}{Q}$.
If $U,V\in\mathcal Q_n$ are the blocks containing $u,v$, respectively, then
\[
U,V\in\operatorname{Out}_{\mathcal Q_n}(P,Q).
\]
Hence $U$ and $V$ are merged by $\CCm$, so $u\equiv_{n+1}v$.
Thus $\tprod{A}{B}$ is contained in a single block of $\Cm(\mathbf S)$, proving coherence.

Let $\mathcal R$ be any $\tau$-multiplication-coherent partition of $\PosId(\mathbf S)$.
Then $\CCm(\mathcal R)=\mathcal R$ by Proposition~\ref{prop:CCx-fixed-points}.
The singleton partition $\mathcal Q_0$ refines $\mathcal R$, so Lemma~\ref{lem:CCx-monotone} yields inductively that every $\mathcal Q_n$ refines $\mathcal R$.
Hence each omega-limit class of $\Cm(\mathbf S)$ is contained in a block of $\mathcal R$.
Thus $\Cm(\mathbf S)$ refines every $\tau$-multiplication-coherent partition, so it is the finest one.
\end{proof}

\begin{theorem}\label{thm:tau-refines-mult}
The $\tau$-cohesive partition $\Ctau(\mathbf S)$ refines the
$\tau$-multiplication-coherent partition $\Cm(\mathbf S)$.
Consequently, every block of $\Cm(\mathbf S)$ is a union of blocks of $\Ctau(\mathbf S)$, and every component of $\Lm(\mathbf S)$ is a union of components of $\Ltau(\mathbf S)$.
\end{theorem}

\begin{proof}
By Theorem~\ref{thm:finest-coherent}, the partition $\Cm(\mathbf S)$ is $\tau$-multiplication-coherent.
By Lemma~\ref{lem:coherent-implies-stable}, it is therefore $\tau$-stable.
Now Theorem~\ref{thm:finest-tau-stable} says that $\Ctau(\mathbf S)$ refines every $\tau$-stable partition.
In particular, it refines $\Cm(\mathbf S)$.
The statements about unions follow by taking unions of the corresponding blocks and their $\tau$-preimages.
\end{proof}

\begin{corollary}\label{cor:tau-cohesive-implies-multiplication-cohesive}
Every $\tau$-cohesive local-unit-aligned ordered semigroup is $\tau$-multiplication-cohesive.
\end{corollary}

\begin{proof}
Let $\mathbf S$ be $\tau$-cohesive.
Then $\Ctau(\mathbf S)=\{\PosId(\mathbf S)\}$.
By Theorem~\ref{thm:tau-refines-mult}, the partition $\Ctau(\mathbf S)$ refines $\Cm(\mathbf S)$.
Since $\Ctau(\mathbf S)$ has only one block, this is possible only if $\Cm(\mathbf S)=\{\PosId(\mathbf S)\}$.
Hence $\mathbf S$ is $\tau$-multiplication-cohesive.
\end{proof}

\begin{example}[$\tau$-multiplication-cohesive need not imply $\tau$-cohesive]
\label{ex:tau-mult-cohesive-not-tau-cohesive}
The converse of Corollary~\ref{cor:tau-cohesive-implies-multiplication-cohesive} is false, even for a finite totally ordered monoid.
Let $S=\{1<2<3<4\}$, and define multiplication by
\[
\begin{array}{c|cccc}
\cdot & 1&2&3&4\\
\hline
1&1&2&3&4\\
2&2&2&4&4\\
3&3&4&4&4\\
4&4&4&4&4
\end{array}
\]
The table is symmetric and each row is nondecreasing, so multiplication is commutative and isotone.
The element $1$ is an identity and $4$ is absorbing; on $\{2,3,4\}$, the only product not equal to $4$ is $2^2=2$, which also makes associativity immediate.
Thus $S$ is a totally ordered monoid.
The local-unit map is
\[
\tau(1)=1,\qquad
\tau(2)=2,\qquad
\tau(3)=1,\qquad
\tau(4)=4.
\]
The positive idempotents are
$
\PosId(\mathbf S)=\{1,2,4\}.
$
A direct calculation gives
\[
\Ctau(\mathbf S)=\{\{1,4\},\{2\}\}.
\]
Indeed, $\tprod{\{1\}}{\{1\}}$ meets $\{4\}$ because $3\cdot3=4$, so the singleton blocks $\{1\}$ and $\{4\}$ are merged in the $\tau$-stable closure.
The block $\{2\}$ remains closed, since products of elements with local unit $2$ again have local unit $2$.
Thus the $\tau$-stable closure is exactly $\{\{1,4\},\{2\}\}$.
Hence the induced $\tau$-cohesive decomposition is
\[
\Ltau(\mathbf S)=\{\{1,3,4\},\{2\}\},
\]
so $\mathbf S$ is not $\tau$-cohesive.
On the other hand, the $\tau$-multiplication-coherent partition is trivial:
\[
\Cm(\mathbf S)=\{\{1,2,4\}\}.
\]
This is forced because products with input local-unit blocks $\{1,4\}$ and $\{2\}$ produce outputs in both $\{1,4\}$ and $\{2\}$, so the two $\tau$-stable blocks must be merged by the multiplication-coherent closure.
Equivalently,
\[
\Lm(\mathbf S)=\{S\}.
\]
Thus $\mathbf S$ is $\tau$-multiplication-cohesive but not $\tau$-cohesive.
\end{example}

\section{Component ordered semigroups}
\label{sec:component-semigroups}

The partitions $\Ctau(\mathbf S)$ and $\Cm(\mathbf S)$ constructed above are partitions of the positive-idempotent skeleton.
Their layer preimages are the corresponding decompositions of the ambient semigroup.
In the infinite setting the correct component object is first of all a semigroup component; a component identity exists only under an additional least-element hypothesis.

\begin{proposition}\label{prop:tau-stable-gives-subsemigroup}
Let $\mathcal P$ be a $\tau$-stable partition of $\PosId(\mathbf S)$, and let $A\in\mathcal P$.
Then $\layer{A}$ is closed under multiplication.
Consequently $\layer{A}$, with the restricted order and multiplication, is an ordered subsemigroup of $\mathbf S$.
\end{proposition}

\begin{proof}
Let $x,y\in\layer{A}$.
Then $\tau(x),\tau(y)\in A$, and $\tau(xy)\in\tprod{A}{A}=A$ because $A$ is $\tau$-stable.
Hence $xy\in\layer{A}$.
\end{proof}

\begin{corollary}\label{cor:canonical-components-subsemigroups}
Every component of $\Ltau(\mathbf S)$ and every component of $\Lm(\mathbf S)$ is an ordered subsemigroup of $\mathbf S$.
\end{corollary}

\begin{proof}
The partition $\Ctau(\mathbf S)$ is $\tau$-stable by Theorem~\ref{thm:finest-tau-stable}.
The partition $\Cm(\mathbf S)$ is $\tau$-multiplication-coherent by Theorem~\ref{thm:finest-coherent}, hence $\tau$-stable by Lemma~\ref{lem:coherent-implies-stable}.
Apply Proposition~\ref{prop:tau-stable-gives-subsemigroup}.
\end{proof}

\begin{proposition}[Intrinsic local units of components over stable blocks]
\label{prop:components-inherit-local-unit-alignment}
Let $\mathcal P$ be a $\tau$-stable partition of $\PosId(\mathbf S)$, and let $A\in\mathcal P$.
Equip $\layer{A}$ with the restricted order and multiplication.
Then $\layer{A}$ is local-unit-aligned.
Moreover,
\[
\PosId(\layer{A})=A
\qquad\text{and}\qquad
\tau_{\layer{A}}=\tau_{\mathbf S}|_{\layer{A}}.
\]
In particular, this applies to every component of $\Lm(\mathbf S)$.
\end{proposition}

\begin{proof}
By Proposition~\ref{prop:tau-stable-gives-subsemigroup}, $\layer{A}$ is an ordered subsemigroup of $\mathbf S$.

First we identify its intrinsic positive idempotents.
If $p\in A$, then $p\in\PosId(\mathbf S)$, and Lemma~\ref{lem:basic_tau}\eqref{item:tau-fixes-idempotents} gives $\tau_{\mathbf S}(p)=p$, so $p\in\layer{A}$.
Since $p$ is ambient positive, it is positive after restricting to $\layer{A}$.
Thus $A\subseteq\PosId(\layer{A})$.

Conversely, let $r\in\PosId(\layer{A})$, where positivity is computed intrinsically in $\layer{A}$.
Put $p:=\tau_{\mathbf S}(r)$.
Since $r\in\layer{A}$, we have $p\in A$, hence $p\in\layer{A}$.
Intrinsic positivity of $r$ applied to $p$ gives
\[
p\le pr
\qquad\text{and}\qquad
p\le rp.
\]
But $p=\tau_{\mathbf S}(r)$ is a two-sided ambient local unit for $r$, so $pr=r=rp$.
Hence $p\le r$.

Now let $z\in S$ be arbitrary.
Since $p\in\PosId(\mathbf S)$, ambient positivity gives $z\le zp$ and $z\le pz$.
Since $p\le r$, isotonicity gives $zp\le zr$ and $pz\le rz$.
Therefore
\[
z\le zr
\qquad\text{and}\qquad
z\le rz.
\]
Thus $r$ is ambient positive.
As $r$ is also idempotent, $r\in\PosId(\mathbf S)$.
By Lemma~\ref{lem:basic_tau}\eqref{item:tau-fixes-idempotents}, $\tau_{\mathbf S}(r)=r$.
Since $r\in\layer{A}$, this gives $r\in A$.
Therefore $\PosId(\layer{A})=A$.

It remains to identify the intrinsic local-unit map.
Let $x\in\layer{A}$, and put $p:=\tau_{\mathbf S}(x)$.
Then $p\in A=\PosId(\layer{A})$, and $xp=x=px$, so $p$ is both an intrinsic right local unit and an intrinsic left local unit of $x$.
If $e\in\layer{A}$ and $xe=x$, then $e$ is an ambient right local unit of $x$, whence $e\le\tau_{\mathbf S}(x)=p$.
Thus $p$ is the greatest intrinsic right local unit of $x$.
The same argument with $ex=x$ shows that $p$ is also the greatest intrinsic left local unit of $x$.
Since $p\in\PosId(\layer{A})$, the restricted semigroup $\layer{A}$ is local-unit-aligned and $\tau_{\layer{A}}(x)=\tau_{\mathbf S}(x)$ for every $x\in\layer{A}$.
\end{proof}

\begin{definition}[Lower-pointed blocks]
\label{def:lower-pointed-block}
A block $A$ of a partition of $\PosId(\mathbf S)$ is called \emph{lower-pointed} if it has a least element.
Since every member of $A$ is a positive idempotent, this least element is also called the least positive idempotent of $A$ and is denoted by $e_A$.
\end{definition}

\begin{proposition}[Lower-pointed blocks and component monoids]
\label{prop:lower-pointed-block-component-monoid}
Let $A$ be a block of a $\tau$-stable partition.
Then $A$ is lower-pointed if and only if the component $\layer{A}$ is a monoid.
In that case the identity of $\layer{A}$ is the least positive idempotent $e_A$ of $A$.
\end{proposition}

\begin{proof}
Suppose first that $A$ is lower-pointed.
For $x\in\layer{A}$, we have $e_A\le\tau(x)$, and Lemma~\ref{lem:basic_tau}\eqref{item:below-tau-gives-unit} gives
\[
xe_A=x=e_Ax.
\]
Thus $e_A$ is the identity of $\layer{A}$.

Conversely, suppose that $\layer{A}$ is a monoid with identity $u_A$.
Since $u_A\in\layer{A}$, also $\tau(u_A)\in A\subseteq\layer{A}$.
The identity law gives $u_A\tau(u_A)=\tau(u_A)$, while $\tau(u_A)$ is a local unit for $u_A$, so $u_A\tau(u_A)=u_A$.
Hence $u_A=\tau(u_A)\in A$.
For every $p\in A$, positivity of $p$ and the identity law in $\layer{A}$ give
\[
u_A\le u_Ap=p.
\]
Therefore $u_A=\min A$, and $A$ is lower-pointed.
\end{proof}

\begin{remark}[Block minima and component minima]
\label{rem:block-minimum-versus-component-minimum}
Lower-pointedness concerns the order on the positive-idempotent block $A$.
It is equivalent to the existence of an identity in $\layer{A}$, but it does not assert that this identity is the least element of the whole component $\layer{A}$.
The latter is a strictly stronger condition, as Example~\ref{ex:skeleton-oriented-identity-not-component-least-lattice} shows.
Thus the component monoids of the finite direct-system decomposition in \cite{JeneiLUARepresentation} arise from lower-pointed blocks, whereas no lower-pointedness assumption is needed for the resolved-transport decomposition below.
\end{remark}

\begin{lemma}[Ambient component calculus]\label{lem:component-calculus}
Let $\mathcal P$ be a $\tau$-stable partition of $\PosId(\mathbf S)$, and let $A\in\mathcal P$.
For all $U,V\subseteq A$, the following hold:
\begin{enumerate}
\item $\layer{U}\subseteq\layer{A}$, and $\layer{U}=\{x\in\layer{A}:\tau_{\mathbf S}(x)\in U\}$;
\item $\layer{U}\cdot\layer{V}\subseteq\layer{A}$;
\item the ambient $\tau$-saturated product satisfies
\[
\tprod{U}{V}=\tau_{\mathbf S}(\layer{U}\cdot\layer{V}),
\]
and the product on the right may be taken inside the restricted ordered subsemigroup $\layer{A}$.
\end{enumerate}
By Proposition~\ref{prop:components-inherit-local-unit-alignment}, the intrinsic local-unit map on $\layer{A}$ is the restriction of the ambient map $\tau_{\mathbf S}$.
\end{lemma}

\begin{proof}
The first item is immediate from the definition of layers.
For the second, if $x\in\layer{U}$ and $y\in\layer{V}$, then $x,y\in\layer{A}$, and Proposition~\ref{prop:tau-stable-gives-subsemigroup} gives $xy\in\layer{A}$.
The third item is the defining ambient formula for $\tprod{U}{V}$, together with the first two items.
\end{proof}

\begin{theorem}\label{thm:Ctau-characterized-by-components}
Let $\mathcal P$ be a $\tau$-stable partition of $\PosId(\mathbf S)$.
Then $\mathcal P=\Ctau(\mathbf S)$ if and only if, for every $A\in\mathcal P$, the component ordered semigroup $\layer{A}$ is intrinsically $\tau$-cohesive; equivalently,
\[
\Ctau(\layer{A})=\{A\},
\]
where $\PosId(\layer{A})=A$ and the intrinsic local-unit map on $\layer{A}$ is the restriction of the ambient $\tau$-map.
\end{theorem}

\begin{proof}
Assume first that $\mathcal P=\Ctau(\mathbf S)$, and let $A\in\mathcal P$.
Suppose that $A$ admits a proper refinement $\mathcal R$ which is $\tau$-stable under products taken inside $\layer{A}$.
Replace $A$ in $\mathcal P$ by the blocks of $\mathcal R$, and call the resulting partition $\mathcal Q$.
By Lemma~\ref{lem:component-calculus}, each block of $\mathcal R$ is also $\tau$-stable when viewed as a subset of $\PosId(\mathbf S)$.
All other blocks of $\mathcal Q$ are blocks of $\mathcal P=\Ctau(\mathbf S)$, hence are $\tau$-stable by Theorem~\ref{thm:finest-tau-stable}.
Thus $\mathcal Q$ is a $\tau$-stable partition of $\PosId(\mathbf S)$, and it properly refines $\Ctau(\mathbf S)$, contradicting the fineness of $\Ctau(\mathbf S)$.
Therefore no such proper refinement of $A$ exists.

Conversely, assume that every component $\layer{A}$, $A\in\mathcal P$, is intrinsically $\tau$-cohesive.
Since $\mathcal P$ is $\tau$-stable and $\Ctau(\mathbf S)$ is the finest $\tau$-stable partition, Theorem~\ref{thm:finest-tau-stable} gives that $\Ctau(\mathbf S)$ refines $\mathcal P$.
Fix $A\in\mathcal P$, and let
\[
\mathcal R_A:=\{C\in\Ctau(\mathbf S):C\subseteq A\}.
\]
Then $\mathcal R_A$ is a partition of $A$.
Each $C\in\mathcal R_A$ is $\tau$-stable in the ambient semigroup, and Lemma~\ref{lem:component-calculus} says that the same $\tau$-stable computation is obtained inside $\layer{A}$.
Hence $\mathcal R_A$ is a $\tau$-stable refinement of $A$ under products taken inside $\layer{A}$.
By the assumed $\tau$-cohesiveness of $\layer{A}$, this refinement is not proper, so $\mathcal R_A=\{A\}$.
Thus every block $A$ of $\mathcal P$ is already a block of $\Ctau(\mathbf S)$.
Since $\Ctau(\mathbf S)$ refines $\mathcal P$, it follows that $\mathcal P=\Ctau(\mathbf S)$.
\end{proof}

The preceding three sections have now supplied the intrinsic pieces of the decomposition.
The partitions live on the positive-idempotent skeleton, while their layer preimages are ordered subsemigroups with the correct intrinsic local-unit map.
The next step is to explain how different components communicate.
This requires two additional pieces of data attached to a multiplication-coherent partition: the quotient join of blocks, and the family of resolved transports indexed by positive idempotents in the receiving block.

\section{Quotient skeletons and resolved transports for coherent partitions}
\label{sec:quotient-skeleton-stage}

Throughout this section, $\mathbf S$ is local-unit-aligned.
A partition $\mathcal P$ of $\PosId(\mathbf S)$ induces the layer decomposition
\[
\mathcal L(\mathcal P):=\{\layer{A}:A\in\mathcal P\},
\qquad
\layer{A}=\tau^{-1}(A).
\]

\begin{definition}[The quotient skeleton of a coherent partition]
\label{def:quotient-skeleton}
Let $\mathcal P$ be a $\tau$-multiplication-coherent partition of $\PosId(\mathbf S)$.
For $A,B\in\mathcal P$, the set $\operatorname{Out}_{\mathcal P}(A,B)$ from Definition~\ref{def:m-output-blocks} is a singleton.
Let
\[
A\vee_{\mathcal P}B
\]
be its unique element; equivalently, $A\vee_{\mathcal P}B$ is the unique block of $\mathcal P$ containing $\tprod{A}{B}$.
Define
\[
A\le_{\mathcal P}B
\qquad\Longleftrightarrow\qquad
A\vee_{\mathcal P}B=B.
\]
We call $(\mathcal P,\le_{\mathcal P},\vee_{\mathcal P})$ the \emph{quotient skeleton} of $\mathcal P$.
The partition $\mathcal P$ is called \emph{block-linear} if $\le_{\mathcal P}$ is a total order.
No totality assumption is imposed unless block-linearity is stated explicitly.
\end{definition}

\begin{lemma}[Semilattice structure of a coherent quotient]
\label{lem:coherent-quotient-semilattice}
Let $\mathcal P$ be a $\tau$-multiplication-coherent partition of $\PosId(\mathbf S)$.
Then $\vee_{\mathcal P}$ is idempotent, associative, and commutative.
Consequently, $\le_{\mathcal P}$, defined by $A\le_{\mathcal P}B$ if and only if $A\vee_{\mathcal P}B=B$, is the join-semilattice order on $\mathcal P$.

Moreover, if $A,B\in\mathcal P$, $x\in\layer{A}$, and $y\in\layer{B}$, then
\[
xy\in\layer{A\vee_{\mathcal P}B}.
\]
In particular, if $A\le_{\mathcal P}B$, $q\in B$, and $x\in\layer{A}$, then $xq\in\layer{B}$.
\end{lemma}

\begin{proof}
By Lemma~\ref{lem:coherent-implies-stable}, the partition $\mathcal P$ is $\tau$-stable.
Hence $A\vee_{\mathcal P}A=A$, proving idempotence.

If $x\in\layer{A}$ and $y\in\layer{B}$, then $\tau(xy)\in\tprod{A}{B}$, and by definition $A\vee_{\mathcal P}B$ is the unique block of $\mathcal P$ meeting $\tprod{A}{B}$.
Thus $xy\in\layer{A\vee_{\mathcal P}B}$.

For associativity, choose $x\in\layer{A}$, $y\in\layer{B}$, and $z\in\layer{C}$.
The preceding paragraph applied to $(xy)z$ puts $\tau((xy)z)$ in $(A\vee_{\mathcal P}B)\vee_{\mathcal P}C$, while the same paragraph applied to $x(yz)$ puts $\tau(x(yz))$ in $A\vee_{\mathcal P}(B\vee_{\mathcal P}C)$.
Since $(xy)z=x(yz)$, these two blocks are equal.

For commutativity, choose $p\in A$ and $q\in B$.
Then $p,q\in\PosId(\mathbf S)$, so Proposition~\ref{prop:positive-idempotents-central} and Lemma~\ref{lem:idempotents-max} give $pq=qp\in\PosId(\mathbf S)$.
Also $p\in\layer{A}$, $q\in\layer{B}$, and $\tau(pq)=pq$.
Thus $pq$ lies in both $A\vee_{\mathcal P}B$ and $B\vee_{\mathcal P}A$, so these blocks are equal.

The standard natural-order argument for a commutative idempotent associative operation gives that $A\le_{\mathcal P}B$ if and only if $A\vee_{\mathcal P}B=B$ is the corresponding join-semilattice order.
The final assertion follows from $q\in\layer{B}$ and the already proved layer formula.
\end{proof}

\begin{definition}[Resolved transports for a coherent partition]
\label{def:resolved-transports-for-partition}
Let $\mathcal P$ be a $\tau$-multiplication-coherent partition of $\PosId(\mathbf S)$.
If $A\le_{\mathcal P}B$, $q\in B$, and $x\in\layer{A}$, define
\[
\lambda^{A,B}_{q}(x):=xq.
\]
By Lemma~\ref{lem:coherent-quotient-semilattice}, this element lies in $\layer{B}$ whenever $A\le_{\mathcal P}B$.
\end{definition}

\begin{theorem}[Product recovery from resolved transports]
\label{thm:resolved-product-recovery}
Let $\mathcal P$ be a $\tau$-multiplication-coherent partition of $\PosId(\mathbf S)$.
Let $A,B\in\mathcal P$, let $x\in\layer{A}$, and let $y\in\layer{B}$.
Put $C:=A\vee_{\mathcal P}B$ in the quotient skeleton of $\mathcal P$.
Then
\[
xy=
\min\{\lambda^{A,C}_{q}(x)\lambda^{B,C}_{q}(y):q\in C\}.
\]
In particular, the quotient skeleton, the induced component ordered semigroups, and the resolved transports of $\mathcal P$ determine the ambient multiplication of $\mathbf S$.
Moreover, in the comparable cases this specializes to the one-sided formulas
\[
xy=
\min\{\lambda^{A,B}_{q}(x)y:q\in B\}
\qquad\text{if }A\le_{\mathcal P}B,
\]
and
\[
xy=
\min\{x\lambda^{B,A}_{q}(y):q\in A\}
\qquad\text{if }B\le_{\mathcal P}A.
\]
\end{theorem}

\begin{proof}
Since $A\le_{\mathcal P}C$ and $B\le_{\mathcal P}C$, the transports $\lambda^{A,C}_{q}(x)$ and $\lambda^{B,C}_{q}(y)$ are defined for every $q\in C$, and both transported elements lie in $\layer{C}$.

Let $q\in C$.
By Proposition~\ref{prop:positive-idempotents-central}, the positive idempotent $q$ is central.
Hence
\[
\lambda^{A,C}_{q}(x)\lambda^{B,C}_{q}(y)
=
(xq)(yq)
=
xyq.
\]
Since $q$ is positive, $xy\le xyq$.
Thus $xy$ is below every element in the displayed set.

On the other hand, $\tau$-multiplication coherence gives $xy\in\layer{A\vee_{\mathcal P}B}=\layer{C}$, so $\tau(xy)\in C$.
Taking $q=\tau(xy)$, we get $xyq=xy\tau(xy)=xy$.
Therefore $xy$ itself belongs to the displayed set and is its least element.

Assume now that $A\le_{\mathcal P}B$, so $C=B$.
For each $q\in B$, centrality gives
\[
\lambda^{A,B}_{q}(x)y=(xq)y=xyq,
\]
and positivity of $q$ gives $xy\le xyq$.
Taking $q=\tau(y)$ gives
\[
\lambda^{A,B}_{\tau(y)}(x)y=x\tau(y)y=xy.
\]
This proves the first one-sided formula.
The second is symmetric: if $B\le_{\mathcal P}A$, then for $q\in A$, positivity gives $xy\le (xq)y$, centrality gives $(xq)y=x(qy)=x\lambda^{B,A}_{q}(y)$, and taking $q=\tau(x)$ gives equality.

The quotient skeleton, the induced component ordered semigroups, and the resolved transports determine the displayed minima.
Hence they determine $xy$ for arbitrary $x,y\in S$, and therefore determine the ambient multiplication of $\mathbf S$.
\end{proof}

\begin{theorem}[Lower-to-higher order recovery from resolved transports]
\label{thm:resolved-order-recovery}
Let $\mathcal P$ be a $\tau$-multiplication-coherent partition of $\PosId(\mathbf S)$.
Let $A,B\in\mathcal P$, let $x\in\layer{A}$, and let $y\in\layer{B}$.
If $A\le_{\mathcal P}B$, then
\[
 x\le y
 \quad\Longleftrightarrow\quad
 \exists q\in B\ \text{such that}\ \lambda^{A,B}_{q}(x)\le y.
\]
Consequently, the quotient skeleton, the induced components, and the resolved transports of $\mathcal P$ determine all lower-to-higher comparisons.
\end{theorem}

\begin{proof}
Assume $A\le_{\mathcal P}B$.
If $x\le y$, then $\tau(y)\in B$, and isotonicity gives
\[
\lambda^{A,B}_{\tau(y)}(x)=x\tau(y)\le y\tau(y)=y.
\]
Conversely, if $\lambda^{A,B}_{q}(x)=xq\le y$ for some $q\in B$, then positivity of $q$ gives $x\le xq\le y$.
The equivalence is therefore an intrinsic lower-to-higher order test for the resolved transport system of $\mathcal P$.
\end{proof}

Applied to $\mathcal P=\Cm(\mathbf S)$, Theorem~\ref{thm:resolved-order-recovery} says that $\mathfrak R(\mathbf S)$ determines all lower-to-higher comparisons in the skeleton $\mathsf S_{\mathrm m}(\mathbf S)$.
Thus product recovery is unconditional, while order recovery is directional at the basic level.
The next section specializes the quotient skeleton and resolved transports to $\Cm(\mathbf S)$, and then records three order-side packages under which the remaining comparisons are also recovered from $\mathfrak R(\mathbf S)$.

\section{The canonical resolved-transport decomposition}
\label{sec:canonical-resolved-transport}

The partition $\Cm(\mathbf S)$ carries a natural block join.
For $A,B\in\Cm(\mathbf S)$, define $A\vee B$ to be the unique block of $\Cm(\mathbf S)$ containing $\tprod{A}{B}$.
This is well-defined by $\tau$-multiplication coherence.

\begin{definition}\label{def:tau-m-block-map}
For $x\in S$, let ${\mathcal C}_{\mathrm{m}}(x)$ denote the unique block of $\Cm(\mathbf S)$ containing $\tau(x)$.
Thus $x\in\layer{A}$ if and only if ${\mathcal C}_{\mathrm{m}}(x)=A$.
\end{definition}

\begin{lemma}[The block join on $\Cm(\mathbf S)$]\label{lem:m-block-product}
The operation $\vee$ on $\Cm(\mathbf S)$ is idempotent, associative, and commutative.
Moreover, for all $x,y\in S$,
\[
{\mathcal C}_{\mathrm{m}}(xy)={\mathcal C}_{\mathrm{m}}(x)\vee{\mathcal C}_{\mathrm{m}}(y).
\]
Equivalently, $\layer{A}\cdot\layer{B}\subseteq\layer{A\vee B}$ for all $A,B\in\Cm(\mathbf S)$.
\end{lemma}

\begin{proof}
Apply Lemma~\ref{lem:coherent-quotient-semilattice} to $\mathcal P=\Cm(\mathbf S)$.
The operation $\vee$ defined above is exactly the quotient operation $\vee_{\mathcal P}$ in that lemma.
Hence $\vee$ is idempotent, associative, and commutative.

Moreover, if $x\in\layer{A}$ and $y\in\layer{B}$, Lemma~\ref{lem:coherent-quotient-semilattice} gives
\[
 xy\in\layer{A\vee B}.
\]
Therefore the block containing $\tau(xy)$, namely ${\mathcal C}_{\mathrm{m}}(xy)$, is $A\vee B$.
Since $A={\mathcal C}_{\mathrm{m}}(x)$ and $B={\mathcal C}_{\mathrm{m}}(y)$, this gives
\[
{\mathcal C}_{\mathrm{m}}(xy)={\mathcal C}_{\mathrm{m}}(x)\vee{\mathcal C}_{\mathrm{m}}(y).
\]
The equivalent inclusion $\layer{A}\cdot\layer{B}\subseteq\layer{A\vee B}$ is the same statement in layer form.
\end{proof}

\begin{definition}\label{def:m-skeleton}
For a local-unit-aligned ordered semigroup $\mathbf S$, define
\[
A\le_{\rm m}B\quad\Longleftrightarrow\quad A\vee B=B
\]
for $A,B\in\Cm(\mathbf S)$.
The structure
\[
\mathsf S_{\mathrm m}(\mathbf S):=(\Cm(\mathbf S),\le_{\rm m},\vee)
\]
is called the \emph{canonical $\tau$-multiplication-coherent skeleton} of $\mathbf S$.
The semigroup $\mathbf S$ is called \emph{$\tau$-block-linear} if the order of
$\mathsf S_{\mathrm m}(\mathbf S)$ is a chain.
\end{definition}

Unless another coherent partition is explicitly under discussion, the terms
\emph{skeleton}, \emph{skeleton block}, and \emph{skeleton order} will henceforth
refer to the canonical skeleton $\mathsf S_{\mathrm m}(\mathbf S)$.

\begin{proposition}[The canonical $\tau$-multiplication-coherent skeleton]
\label{prop:m-skeleton-join}
For every local-unit-aligned ordered semigroup $\mathbf S$, the canonical $\tau$-multiplication-coherent skeleton $\mathsf S_{\mathrm m}(\mathbf S)$ is a join-semilattice.
Moreover, if $\mathbf S$ is positive-idempotent-linear, then $\mathbf S$ is $\tau$-block-linear.
In particular, every local-unit-aligned totally ordered semigroup is $\tau$-block-linear, but totality of the whole ambient order is not needed for this conclusion.
\end{proposition}

\begin{proof}
By Lemma~\ref{lem:m-block-product}, $\vee$ is associative, commutative, and idempotent.
Hence it is the join operation for the displayed natural order.
Now assume that $\Sk(\mathbf S)$ is a chain.
For $A,B\in\Cm(\mathbf S)$, choose $p\in A$ and $q\in B$.
By Lemma~\ref{lem:idempotents-max}, $pq=qp=\max\{p,q\}$.
Since $p\in\layer{A}$, $q\in\layer{B}$, and $pq$ is a positive idempotent, Lemma~\ref{lem:basic_tau} gives $pq=\tau(pq)\in\tprod{A}{B}$.
Thus $pq\in A\vee B$.
But $\max\{p,q\}\in A\cup B$, so $A\vee B$ is either $A$ or $B$.
Therefore any two skeleton blocks are comparable, proving $\tau$-block-linearity.
\end{proof}

\begin{definition}[Canonical resolved-transport homomorphisms]
\label{def:canonical-resolved-transport-homomorphisms}
Let $A,B\in\Cm(\mathbf S)$ with $A\le_{\rm m}B$, and let $q\in B$.
The \emph{canonical resolved-transport homomorphism} indexed by $q$ is the map
\[
\lambda^{A,B}_{q}\colon \layer{A}\to\layer{B},
\qquad
\lambda^{A,B}_{q}(x):=xq.
\]
The \emph{canonical resolved-transport system} of $\mathbf S$ is
\[
\mathfrak R(\mathbf S):=
\Bigl(
(\Cm(\mathbf S),\le_{\rm m},\vee),
\bigl(\layer{A},\le|_{\layer{A}},\cdot|_{\layer{A}}\bigr)_{A\in\Cm(\mathbf S)},
\bigl(\lambda^{A,B}_{q}\bigr)_{A\le_{\rm m}B,\ q\in B}
\Bigr).
\]
By Proposition~\ref{prop:components-inherit-local-unit-alignment}, the component $\layer{A}$ is intrinsically local-unit-aligned, with
\[
\PosId(\layer{A})=A
\qquad\text{and}\qquad
\tau_{\layer{A}}=\tau|_{\layer{A}}.
\]
Thus the index $q\in B$ is equivalently a positive idempotent of the target component $\layer{B}$.
\end{definition}

\begin{proposition}[Basic properties of the resolved transports in $\mathfrak R(\mathbf S)$]\label{prop:resolved-transport-basic}
Let $A,B,C\in\Cm(\mathbf S)$.
\begin{enumerate}
\item\label{item:ct-well-defined} If $A\le_{\rm m}B$, $q\in B$, and $x\in\layer{A}$, then $\lambda^{A,B}_{q}(x)\in\layer{B}$.
\item\label{item:ct-reflexive} If $x\in\layer{A}$, then $\lambda^{A,A}_{\tau(x)}(x)=x$.
\item\label{item:ct-lax-comp} If $A\le_{\rm m}B\le_{\rm m}C$, $q\in B$, $r\in C$, and $x\in\layer{A}$, then
\[
\lambda^{B,C}_{r}\bigl(\lambda^{A,B}_{q}(x)\bigr)
=
\lambda^{A,C}_{\lambda^{B,C}_{r}(q)}(x).
\]
\item\label{item:ct-hom} If $A\le_{\rm m}B$ and $q\in B$, then $\lambda^{A,B}_{q}\colon\layer{A}\to\layer{B}$ is an isotone semigroup homomorphism.
\item\label{item:ct-growing} If $A\le_{\rm m}B$, $p,q\in B$, and $p\le q$, then, for all $x\in\layer{A}$,
\[
\lambda^{A,B}_{p}(x)\le \lambda^{A,B}_{q}(x),
\qquad
\lambda^{A,B}_{q}(x)=\lambda^{A,B}_{p}(x)q.
\]
\item\label{item:ct-local-factor} If $A\le_{\rm m}B$, $y\in\layer{B}$, $p\in B$, and $p\le\tau(y)$, then, for all $x\in\layer{A}$,
\[
\lambda^{A,B}_{p}(x)y
=
\lambda^{A,B}_{\tau(y)}(x)y
=
xy.
\]
\end{enumerate}
\end{proposition}
\begin{proof}
For item~\eqref{item:ct-well-defined}, let $q\in B$.
Since $q\in\PosId(\mathbf S)$, Lemma~\ref{lem:basic_tau}\eqref{item:tau-fixes-idempotents} gives $q\in\layer{B}$.
By Lemma~\ref{lem:m-block-product}, $xq\in\layer{A\vee B}=\layer{B}$.
Item~\eqref{item:ct-reflexive} follows from $x=x\tau(x)$ and $\tau(x)\in A$.
For item~\eqref{item:ct-lax-comp}, compute
$
\lambda^{B,C}_{r}\bigl(\lambda^{A,B}_{q}(x)\bigr)
=(xq)r=x(qr).
$
Also $\lambda^{B,C}_{r}(q)=qr$.
By Lemma~\ref{lem:idempotents-max}, $qr\in\PosId(\mathbf S)$, while
Lemma~\ref{lem:m-block-product} gives $qr\in\layer{C}$.
Hence $qr=\tau(qr)\in C$, so the right-hand side is
$\lambda^{A,C}_{\lambda^{B,C}_{r}(q)}(x)$.
For item~\eqref{item:ct-hom}, isotonicity is inherited from multiplication.
If $x,z\in\layer{A}$, then centrality and idempotence of $q$ give
$
\lambda^{A,B}_{q}(xz)=xzq=(xq)(zq)=\lambda^{A,B}_{q}(x)\lambda^{A,B}_{q}(z).
$
For item~\eqref{item:ct-growing}, if $p\le q$, then $xp\le xq$ by isotonicity.
Lemma~\ref{lem:idempotents-max} gives $pq=q$, and hence
$
\lambda^{A,B}_{p}(x)q=(xp)q=x(pq)=xq=\lambda^{A,B}_{q}(x).
$
For item~\eqref{item:ct-local-factor}, Lemma~\ref{lem:basic_tau}\eqref{item:below-tau-gives-unit} gives $yp=y=py$ whenever $p\le\tau(y)$.
Therefore
$
\lambda^{A,B}_{p}(x)y=(xp)y=x(py)=xy.
$
The same computation with $p=\tau(y)$ gives $\lambda^{A,B}_{\tau(y)}(x)y=xy$.
\end{proof}

The homomorphisms $\lambda^{A,B}_{q}$ are the primitive connecting homomorphisms of $\mathfrak R(\mathbf S)$.
For a fixed right factor $y\in\layer{B}$, the homomorphism indexed by the actual local unit $\tau(y)$ is sufficient to compute $xy$, but it is not unique for this purpose: every smaller positive idempotent $p\le\tau(y)$ gives the same product after multiplication by $y$.
Indeed, if
\[
D(p):=\{y\in\layer{B}:p\le\tau(y)\},
\]
then $p\le q$ implies $D(q)\subseteq D(p)$.
The price is that the transported elements themselves may differ; for $p\le q$, one has $\lambda^{A,B}_{p}(x)\le\lambda^{A,B}_{q}(x)$, and the larger transport is obtained from the smaller one by multiplying once more by $q$.

Applied to $\mathcal P=\Cm(\mathbf S)$,
Theorem~\ref{thm:resolved-product-recovery} shows that
$\mathfrak R(\mathbf S)$ determines the ambient multiplication.
In this specialization, $A\vee_{\mathcal P}B$ is the block join
$A\vee B$, and $A\le_{\mathcal P}B$ is the canonical skeleton order
$A\le_{\rm m}B$.

Likewise, Theorem~\ref{thm:resolved-order-recovery} shows that
$\mathfrak R(\mathbf S)$ determines every comparison from a lower
canonical block to a higher one.
This directional recovery is the strongest unconditional order statement:
the same canonical resolved-transport system may support different reverse
cross-block comparisons, and hence different ambient orders.
The following example isolates this obstruction already over a two-element
canonical skeleton.

\begin{example}[The canonical resolved-transport system does not determine the full ambient order]
\label{ex:strict-upper-shadow-ambiguity}
Let $\mathbb N=\{0,1,2,\ldots\}$, and put
\[
A=\{a_n:n\in\mathbb N\},
\qquad
B=\{b_n:n\in\mathbb N\},
\qquad
S=A\cup B.
\]
Define a commutative multiplication on $S$ by
\[
a_i a_j=a_{i+j},
\qquad
a_i b_j=b_{i+j},
\qquad
b_i b_j=b_{i+j}.
\]
Thus $a_0$ and $b_0$ are the respective identities of the $A$- and
$B$-parts, and a product lies in $B$ as soon as one factor does.

Give both components their natural chain orders:
$a_i\le_A a_j$ and $b_i\le_B b_j$ exactly when $i\le j$.
Define two ambient partial orders extending these component orders.
For $\le_0$, set
\[
a_i\le_0 b_j
\quad\Longleftrightarrow\quad
i\le j,
\]
with no cross-comparisons from $B$ to $A$.
For $\le_1$, set
\[
a_i\le_1 b_j
\quad\Longleftrightarrow\quad
i\le j,
\qquad
b_i\le_1 a_j
\quad\Longleftrightarrow\quad
i<j.
\]
Thus $\le_1$ interleaves the components as
$a_0<b_0<a_1<b_1<a_2<b_2<\cdots$.

Both $\langle S,\le_0,\cdot\rangle$ and
$\langle S,\le_1,\cdot\rangle$ are ordered semigroups: multiplication
adds indices and moves the product into $B$ whenever one factor lies in
$B$, so it preserves each of the preceding order clauses.
In both structures the positive idempotents are $a_0$ and $b_0$, with
$\tau(a_n)=a_0$ and $\tau(b_n)=b_0$ for every $n\in\mathbb N$.
Hence both structures are local-unit-aligned.

For either order, the discrete partition of the positive-idempotent skeleton
has blocks
\[
U=\{a_0\},
\qquad
V=\{b_0\},
\]
and is $\tau$-multiplication-coherent, since
\[
\tprod{U}{U}=U,
\qquad
\tprod{U}{V}=V,
\qquad
\tprod{V}{V}=V.
\]
It is therefore the canonical partition $\Cm(\mathbf S)$.
Its canonical skeleton is the chain $U<V$, and its components are the same
ordered semigroups $A$ and $B$ in both structures.
The resolved transport is also the same:
\[
\lambda^{U,V}_{b_0}(a_n)=a_n b_0=b_n
\qquad(n\in\mathbb N).
\]

Nevertheless, the resulting ambient orders differ.
Take $u=a_1\in\layer{U}$ and $v=b_0\in\layer{V}$.
Then
\[
v=b_0<_B b_1=\lambda^{U,V}_{b_0}(u),
\]
but $u$ and $v$ are incomparable under $\le_0$, whereas
$v=b_0<_1a_1=u$.
Thus the same strict upper-shadow signal gives incomparability in the first
structure and the corresponding reverse comparison in the second.

The two structures therefore have the same underlying semigroup, canonical
skeleton, component ordered semigroups, and resolved transports, but different
ambient orders.
Consequently, the canonical resolved-transport system does not by itself
determine the full ordering of the semigroup.
It determines the lower-to-higher comparisons, but not whether an upper
element lying strictly below the transported image of a lower element is
below that lower element or is incomparable with it.
\end{example}

Accordingly, full order recovery requires additional information about the
cross-component order.
There are two distinct issues.
Comparisons between incomparable canonical skeleton blocks must first be
excluded or otherwise specified, while reverse comparisons between comparable
blocks require an additional decision rule.
The skeleton-supported condition addresses the first issue.
The subsequent order-recovery packages address the second in different ways.

\begin{definition}[Skeleton-supported order]\label{def:skeleton-supported-order}
The ambient order is \emph{skeleton-supported} if, for all $A,B\in\Cm(\mathbf S)$, $x\in\layer{A}$, and $y\in\layer{B}$,
\[
 x\le y
 \quad\Longrightarrow\quad
 A\text{ and }B\text{ are comparable in }\le_{\rm m}.
\]
Equivalently, there are no ambient-order comparisons between elements lying in incomparable skeleton blocks.
This is an order-side support condition; it is not needed for multiplication.
\end{definition}

\subsection{The skeleton-oriented case}

\begin{definition}[Skeleton-oriented order]\label{def:skeleton-oriented-order}
We say that the ambient order is \emph{skeleton-oriented} if, for all $A,B\in\Cm(\mathbf S)$, $x\in\layer{A}$, and $y\in\layer{B}$,
\[
 x\le y
 \quad\Longrightarrow\quad
 A\le_{\rm m}B.
\]
Equivalently, every ambient-order comparison between distinct skeleton blocks is oriented from the lower skeleton block to the higher skeleton block.
This condition implies skeleton-supportedness.
\end{definition}

\begin{corollary}[Full order recovery under skeleton orientation]
\label{cor:oriented-full-order-recovery}
If $\mathbf S$ is skeleton-oriented, then $\mathfrak R(\mathbf S)$ determines the full ambient order.
Explicitly, for $x\in\layer{A}$ and $y\in\layer{B}$, the relation $x\le y$ is recovered as follows:
\begin{enumerate}
\item if $A=B$, use the order of the component $\layer{A}$;
\item if $A<_{\rm m}B$, use the lower-to-higher test of Theorem~\ref{thm:resolved-order-recovery};
\item if $A\not\le_{\rm m}B$, then $x\nleq y$.
\end{enumerate}
\end{corollary}

\begin{proof}
The first case is part of the component data.
The second case is exactly Theorem~\ref{thm:resolved-order-recovery}.
In the third case, skeleton orientation would force $A\le_{\rm m}B$ from $x\le y$, contradicting $A\not\le_{\rm m}B$.
These three cases exhaust all pairs of blocks.
\end{proof}

Totality gives skeleton orientation an additional endpoint consequence that is unavailable in general partial orders.
The next proposition records precisely the receiving components for which this consequence is relevant.

\begin{proposition}[Totality and skeleton orientation make receiving-component identities least]\label{prop:totality-orientation-least-receiving-identities}
Let $\mathbf S$ be a local-unit-aligned totally ordered semigroup.
Assume that the ambient order is skeleton-oriented.
Let $B\in\Cm(\mathbf S)$ be a block with a lower skeleton block below it, meaning that there exists $A\in\Cm(\mathbf S)$ with $A<_{\rm m}B$.
Suppose that $\layer{B}$ is a monoid with identity $e_B$.
Then $e_B$ is the least element of $\layer{B}$.
\end{proposition}

\begin{proof}
First observe that $e_B$ is an ambient positive idempotent.
Since $e_B\in\layer{B}$, we have $\tau(e_B)\in B$.
Also $\tau(e_B)\in\layer{B}$, because $B\subseteq\PosId(\mathbf S)$.
As $e_B$ is the identity of the monoid $\layer{B}$, we get
$e_B\tau(e_B)=\tau(e_B)$.
On the other hand, $\tau(e_B)$ is a two-sided local unit for $e_B$, so
$e_B\tau(e_B)=e_B$.
Thus $e_B=\tau(e_B)$, and hence $e_B\in B\subseteq\PosId(\mathbf S)$.

Choose $p\in A$.
Since $p$ is a positive idempotent, Lemma~\ref{lem:basic_tau}\eqref{item:tau-fixes-idempotents} gives $\tau(p)=p$, and hence $p\in\layer{A}$.
Let $y\in\layer{B}$.
Because the ambient order is total, either $p\le y$ or $y\le p$.
The second alternative is impossible by skeleton orientation: it would give a cross-block comparison from $B$ down to $A$, contrary to $A<_{\rm m}B$.
Thus $p\le y$.
In particular, applying this to $y=e_B$, we get $p\le e_B$.
Since both $p$ and $e_B$ are positive idempotents, Lemma~\ref{lem:idempotents-max} gives $pe_B=e_B$.
Multiplying $p\le y$ by $e_B$, we obtain $pe_B\le ye_B$.
Since $e_B$ is the identity of $\layer{B}$, this gives $e_B\le y$.
Therefore $e_B$ is the least element of $\layer{B}$.
\end{proof}

The totality assumption is essential.
The following finite lattice-ordered example is skeleton-oriented, but the identity of its receiving component is not least.

\begin{example}[Skeleton orientation does not force a component identity to be least]
\label{ex:skeleton-oriented-identity-not-component-least-lattice}
Skeleton orientation alone does not force the identity of a receiving component to be least.
This can already happen on a finite distributive lattice.

Let $S=\{\bot,c,a,e\}$ be ordered as the four-element Boolean lattice:
$\bot<c<e$, $ \bot<a<e$, with $a$ and $c$ incomparable.
Define a commutative multiplication by
\[
\begin{array}{c|cccc}
\cdot & \bot & c & a & e\\
\hline
\bot & \bot & \bot & a & a\\
c    & \bot & c    & a & e\\
a    & a    & a    & a & a\\
e    & a    & e    & a & e
\end{array}
\]
This multiplication is associative.
Indeed, it is the product associated with an ordinary two-level direct system, obtained from the lower component $\{\bot,c\}$, with identity $c$, the upper component $\{a,e\}$, with identity $e$, and the transition homomorphism
$\rho(\bot)=a$, $ \rho(c)=e$.
The order is compatible with multiplication, as is checked on the cover relations
$\bot<c$, $ \bot<a$, $ c<e$, $ a<e$.

The positive idempotents are exactly $c$ and $e$.
The element $c$ is the identity of all of $S$, hence is positive.
The element $e$ is positive because
$\bot e=a\ge \bot$, $ ce=e\ge c$, $ ae=a$, $ e^2=e$.
The element $\bot$ is not positive, since $c\bot=\bot<c$, and $a$ is not positive, since $ea=a<e$.
The local-unit map is
$\tau(\bot)=c$, $ \tau(c)=c$, $ \tau(a)=e$, $ \tau(e)=e$.
Thus
$\layer{c}=\{\bot,c\}$, $ \layer{e}=\{a,e\}$.
Moreover,
$\layer{c}\layer{c}\subseteq\layer{c},
$ $
\layer{c}\layer{e}\cup\layer{e}\layer{c}\subseteq\layer{e},
$ $
\layer{e}\layer{e}\subseteq\layer{e}$.
Hence
$\Cm(\mathbf S)=\bigl\{\{c\},\{e\}\bigr\},
$ $
\{c\}<_{\rm m}\{e\}$.
Thus these two exact layers are precisely the canonical components.

The upper component $\layer{e}$ has a lower skeleton block below it, but its identity $e$ is not least in $\layer{e}$, since $a<e$.
Nevertheless the ambient order is skeleton-oriented.
The only cross-block comparisons are
$\bot<a$, $ \bot<e$, $ c<e$,
all from $\layer{c}$ to $\layer{e}$.
There is no comparison from $\layer{e}$ down to $\layer{c}$.

Thus skeleton orientation does not, by itself, force the identity of a receiving component to be least in that component.
What fails here is the finite rigid directed-lexicographic mechanism: the transition homomorphism is not unit-constant, since
$\rho(\bot)=a$, $ \rho(c)=e$.
In the finite rigid directed-lexicographic situation, an inequality $a<e=\rho(c)$ would create a reverse comparison from the upper component to the lower one.
In this lattice order, that reverse comparison is simply absent.
\end{example}

Thus skeleton orientation controls the direction of cross-block comparisons, but without totality it does not determine the internal endpoint position of a receiving-component identity.

\begin{remark}[$\tau$-block-linearity is not an order-recovery axiom]\label{rem:block-linearity-not-order-recovery}
The join formula in Theorem~\ref{thm:resolved-product-recovery} recovers
multiplication without assuming that $\mathbf S$ is
$\tau$-block-linear.
The condition of $\tau$-block-linearity only removes incomparable
canonical skeleton blocks from the order bookkeeping.
It does not by itself rule out reverse cross-block comparisons, that is, comparisons $y\le x$ with $x\in\layer{A}$, $y\in\layer{B}$, and $A<_{\rm m}B$.
Nor does it decide which of these reverse comparisons hold.
Thus full order recovery needs an additional cross-order principle, such as skeleton orientation, skeleton-supported cross-block totality along the skeleton, skeleton-supported strict upper-shadow reflection, or a comparable explicit rule.
In a totally ordered semigroup, positive-idempotent linearity implies
$\tau$-block-linearity by Proposition~\ref{prop:m-skeleton-join}; after the lower-to-higher comparison has been tested for two distinct blocks, trichotomy decides the opposite orientation.
This is the chain special case captured by cross-block totality below.
\end{remark}

\subsection{The cross-block total case}

\begin{definition}[Cross-block totality]\label{def:cross-block-totality}
We say that the ambient order is \emph{cross-block total along the skeleton} if, whenever $A<_{\rm m}B$, $x\in\layer{A}$, and $y\in\layer{B}$, the elements $x$ and $y$ are comparable in the ambient order.
No totality is required inside a single component, and this condition makes no assertion about pairs of incomparable skeleton blocks.
\end{definition}

\begin{proposition}[Full order recovery under cross-block totality]\label{prop:cross-block-totality-recovery}
If $\mathbf S$ is skeleton-supported and cross-block total along the skeleton, then $\mathfrak R(\mathbf S)$ determines the full ambient order.
Explicitly, for $x\in\layer{A}$ and $y\in\layer{B}$, the relation $x\le y$ is recovered as follows:
\begin{enumerate}
\item if $A=B$, use the order of the component $\layer{A}$;
\item if $A<_{\rm m}B$, use the lower-to-higher test of Theorem~\ref{thm:resolved-order-recovery};
\item if $B<_{\rm m}A$, then $x\le y$ holds exactly when the lower-to-higher test does not give $y\le x$;
\item if $A$ and $B$ are incomparable, then $x\nleq y$.
\end{enumerate}
\end{proposition}

\begin{proof}
The first case is component data, and the second is Theorem~\ref{thm:resolved-order-recovery}.
Assume $B<_{\rm m}A$.
The resolved-transport system can test whether $y\le x$, since this is a lower-to-higher comparison.
If $y\le x$, then $x\nleq y$ by antisymmetry: the layers $\layer{A}$ and $\layer{B}$ are disjoint, so $x\ne y$.
If $y\nleq x$, then cross-block totality forces $x\le y$.
Finally, if $A$ and $B$ are incomparable, skeleton-supportedness gives $x\nleq y$.
These cases exhaust all pairs of blocks.
\end{proof}

\subsection{The strict upper-shadow reflection case}

Example~\ref{ex:strict-upper-shadow-ambiguity} shows that the component order
can detect a strict upper shadow without the resolved-transport system
determining whether it represents an actual reverse comparison.
Strict upper-shadow reflection resolves precisely this ambiguity by requiring
the internal strict-shadow test to be equivalent to the corresponding reverse
ambient comparison.

\begin{definition}[Strict upper-shadow reflection]\label{def:strict-upper-shadow-reflection}
For $B\in\Cm(\mathbf S)$, write $<_{B}$ for the strict part of the component order on $\layer{B}$.
We say that the ambient order is \emph{strict upper-shadow reflective along the skeleton} if, whenever $A<_{\rm m}B$, $x\in\layer{A}$, and $y\in\layer{B}$,
\[
 y\le x
 \quad\Longleftrightarrow\quad
 y<_{B}\lambda^{A,B}_{q}(x)\text{ for every }q\in B.
\]
The right-hand condition is computed entirely inside the upper component $\layer{B}$.
This condition controls reverse comparisons along comparable skeleton blocks only.
\end{definition}

The implication from an ambient reverse comparison to the strict upper-shadow
condition is automatic. Indeed, let $A<_{\rm m}B$,
$x\in\layer{A}$, and $y\in\layer{B}$. If $y\le x$, then, for every
$q\in B$,
\[
y\le x\le xq=\lambda_q^{A,B}(x),
\]
since $q$ is positive. Equality
$y=\lambda_q^{A,B}(x)$ would give $x\le y$, hence $x=y$, contrary to
the disjointness of the components. Therefore
$y<_{B}\lambda_q^{A,B}(x)$ for every $q\in B$.
Thus the substantive content of strict upper-shadow reflection is the converse
implication
\[
\bigl[y<_{B}\lambda_q^{A,B}(x)\text{ for every }q\in B\bigr]
\Longrightarrow y\le x.
\]

\begin{proposition}[Full order recovery under strict upper-shadow reflection]\label{prop:strict-upper-shadow-recovery}
If $\mathbf S$ is skeleton-supported and strict upper-shadow reflective along the skeleton, then $\mathfrak R(\mathbf S)$ determines the full ambient order.
Explicitly, for $x\in\layer{A}$ and $y\in\layer{B}$, the relation $x\le y$ is recovered as follows:
\begin{enumerate}
\item if $A=B$, use the order of the component $\layer{A}$;
\item if $A<_{\rm m}B$, use the lower-to-higher test of Theorem~\ref{thm:resolved-order-recovery};
\item if $B<_{\rm m}A$, test whether $x<_{A}\lambda^{B,A}_{q}(y)$ for every $q\in A$;
\item if $A$ and $B$ are incomparable, then $x\nleq y$.
\end{enumerate}
\end{proposition}

\begin{proof}
The first case is component data, and the second is Theorem~\ref{thm:resolved-order-recovery}.
For the third case, $B<_{\rm m}A$, the asserted criterion is exactly strict upper-shadow reflection applied with lower block $B$, upper block $A$, lower element $y$, and upper element $x$.
All terms in the test belong to $\mathfrak R(\mathbf S)$.
If $A$ and $B$ are incomparable, skeleton-supportedness gives $x\nleq y$.
These cases exhaust all pairs of blocks.
\end{proof}

The hypothesis is genuinely additional: the two structures in
Example~\ref{ex:strict-upper-shadow-ambiguity} have identical canonical
resolved-transport data, but only one realizes the strict upper-shadow signal
as a reverse ambient comparison.

\subsection{Order recovery from componentwise order duality}

\begin{definition}[Component-preserving order anti-automorphism]
\label{def:component-preserving-order-antiautomorphism}
Let $\mathbf S$ be a local-unit-aligned ordered semigroup, and let
\[
\nu\colon S\longrightarrow S
\]
be an anti-automorphism of the ordered set $\langle S,\le\rangle$, that is, a bijection satisfying
\[
x\le y
\quad\Longleftrightarrow\quad
\nu(y)\le\nu(x)
\qquad(x,y\in S).
\]
Here ``anti-automorphism'' refers only to the ordered-set structure; no compatibility of $\nu$ with multiplication is assumed or needed.
We call $\nu$ \emph{component-preserving} if
\[
{\mathcal C}_{\mathrm m}(\nu(x))
=
{\mathcal C}_{\mathrm m}(x)
\qquad(x\in S).
\]
Equivalently, every component $\layer{A}$ of $\Lm(\mathbf S)$ is invariant under $\nu$.

The stronger identity
\[
\tau(\nu(x))=\tau(x)
\qquad(x\in S)
\]
implies component preservation.

The corresponding \emph{order-dual resolved-transport system} is
\[
\mathfrak R_{\nu}(\mathbf S)
:=
\left(
\mathfrak R(\mathbf S),
\bigl(\nu\!\upharpoonright_{\layer{A}}\bigr)_{A\in\Cm(\mathbf S)}
\right).
\]
\end{definition}

\begin{proposition}[Full order recovery from componentwise order duality]
\label{prop:antiautomorphism-full-order-recovery}
Assume that $\mathbf S$ is skeleton-supported and carries a component-preserving order anti-automorphism $\nu$.
Then $\mathfrak R_{\nu}(\mathbf S)$ reconstructs the full ambient order.
Explicitly, for $x\in\layer{A}$ and $y\in\layer{B}$, the relation $x\le y$ is recovered as follows:
\begin{enumerate}
\item if $A=B$, use the order of the component $\layer{A}$;
\item if $A<_{\rm m}B$, use the lower-to-higher test of Theorem~\ref{thm:resolved-order-recovery};
\item if $B<_{\rm m}A$, use
\[
x\le y
\quad\Longleftrightarrow\quad
\nu(y)\le\nu(x)
\quad\Longleftrightarrow\quad
\exists q\in A\text{ such that }
\lambda^{B,A}_{q}(\nu(y))\le\nu(x);
\]
\item if $A$ and $B$ are incomparable, then $x\nleq y$.
\end{enumerate}
Consequently, $\mathfrak R_{\nu}(\mathbf S)$ reconstructs the full expanded ordered semigroup
\[
\langle S,\le,\cdot,\nu\rangle.
\]
\end{proposition}

\begin{proof}
The first case is part of the component data, and the second is Theorem~\ref{thm:resolved-order-recovery}.
Suppose that $B<_{\rm m}A$.
Since $\nu$ is an order anti-automorphism,
\[
x\le y
\quad\Longleftrightarrow\quad
\nu(y)\le\nu(x).
\]
Component preservation gives
\[
\nu(y)\in\layer{B},
\qquad
\nu(x)\in\layer{A}.
\]
Thus $\nu(y)\le\nu(x)$ is a lower-to-higher comparison, and Theorem~\ref{thm:resolved-order-recovery} gives
\[
\nu(y)\le\nu(x)
\quad\Longleftrightarrow\quad
\exists q\in A\text{ such that }
\lambda^{B,A}_{q}(\nu(y))\le\nu(x).
\]
If $A$ and $B$ are incomparable, skeleton-supportedness excludes $x\le y$.
Hence every ambient-order comparison is reconstructed.

The multiplication is reconstructed by Theorem~\ref{thm:resolved-product-recovery}, and $\nu$ is recorded componentwise in $\mathfrak R_{\nu}(\mathbf S)$.
Therefore the full expanded ordered semigroup is reconstructed.
\end{proof}

\begin{corollary}[$\tau$-block-linear case]
\label{cor:antiautomorphism-chain-skeleton-recovery}
Let $\mathbf S$ carry a component-preserving order anti-automorphism $\nu$.
If $\mathbf S$ is $\tau$-block-linear, then
$\mathfrak R_{\nu}(\mathbf S)$ reconstructs the full expanded ordered semigroup
\[
\langle S,\le,\cdot,\nu\rangle.
\]
\end{corollary}

\begin{proof}
Since $\mathbf S$ is $\tau$-block-linear, its canonical skeleton has no
incomparable blocks, so skeleton-supportedness is automatic.
Proposition~\ref{prop:antiautomorphism-full-order-recovery} applies.
\end{proof}

\begin{remark}[Balanced involutive residuated expansions]
\label{rem:involutive-order-recovery}
Suppose that $\mathbf S$ is the ordered-semigroup reduct of a balanced involutive residuated po-semigroup, with residuals $\backslash$ and $/$ and negations $x^{\sim}$ and $x^{-}$.
The involutive terminology and the contraposition identities used below are as in \cite[Section~3.3 and Lemma~3.17(1)--(2)]{GalatosJipsenKowalskiOno2007}.

By Remark~\ref{rem:relation-to-finite-monoid}, the reduct is local-unit-aligned and
\[
\tau(x)=x\backslash x=x/x.
\]
Using the contraposition identities, we obtain
\[
\begin{aligned}
\tau(x^{\sim})
&=x^{\sim}\backslash x^{\sim}
 =x^{\sim}/x^{\sim}
 =x\backslash x
 =\tau(x),\\
\tau(x^{-})
&=x^{-}/x^{-}
 =x^{-}\backslash x^{-}
 =x/x
 =\tau(x).
\end{aligned}
\]
Thus both negations preserve every exact $\tau$-layer and hence every component of $\Lm(\mathbf S)$.

Since the negations are mutually inverse antitone maps, either of them is a component-preserving order anti-automorphism.
Proposition~\ref{prop:antiautomorphism-full-order-recovery} therefore shows that, under skeleton-supportedness, adjoining the componentwise action of either negation to the resolved-transport data recovers the full ambient order.
If $\mathbf S$ is $\tau$-block-linear, skeleton-supportedness is automatic.

Finally, either negation determines the other, and multiplication together with the recovered order determines both residuals.
Hence either augmented resolved-transport system reconstructs the full balanced involutive residuated expansion.
\end{remark}

\section{Ordinary direct systems from skeleton-monotone least-element collapses}
\label{sec:least-element-collapse}

The previous section gives the resolved-transport form of the decomposition--reconstruction procedure.
We now separate the two ingredients needed for the ordinary single-map picture.
Lower-pointedness of each block receiving a proper transition selects its least positive idempotent and hence one pointwise least target map.
Skeleton-monotonicity of these receiving minima makes the selected maps compatible under composition.
Together, these hypotheses recover an ordinary single-index direct system.
Throughout this section we use the canonical skeleton order from Definition~\ref{def:m-skeleton}.

\begin{definition}
\label{def:least-element-collapse}
Put
\[
\mathcal R
:=
\bigl\{B\in\Cm(\mathbf S): A<_{\rm m}B\text{ for some }A\in\Cm(\mathbf S)\bigr\}.
\]
Thus $\mathcal R$ is the set of blocks receiving a proper transition.
Assume that every block $B\in\mathcal R$ is lower-pointed, and write $e_B:=\min B$.
We say that the receiving least idempotents are \emph{skeleton-monotone} if
\[
B\le_{\rm m}C
\quad\Longrightarrow\quad
e_B\le e_C
\qquad(B,C\in\mathcal R).
\]
When this holds, define
\[
\rho_{A,A}:=\mathrm{id}_{\layer{A}}
\qquad(A\in\Cm(\mathbf S))
\]
and, for $A<_{\rm m}B$, define
\[
\rho_{A,B}(x):=xe_B
\qquad(x\in\layer{A}).
\]
We call $(\rho_{A,B})_{A\le_{\rm m}B}$ the \emph{least-element collapse} of the resolved-transport system.
\end{definition}

Thus lower-pointedness alone supplies the candidate least-target maps
\[
x\longmapsto xe_B.
\]
The term \emph{least-element collapse} is used here only after the receiving minima have also been assumed skeleton-monotone; this second hypothesis is what permits the direct-system composition law proved below.

\begin{remark}[When receiving-block monotonicity is automatic]
\label{rem:skeleton-monotonicity-automatic-under-linearity}
The skeleton-monotonicity hypothesis in Definition~\ref{def:least-element-collapse} is automatic under positive-idempotent linearity, hence in particular under ambient totality, once every receiving block is lower-pointed.

Indeed, assume that $\PosId(\mathbf S)$ is totally ordered and let $B,C\in\mathcal R$, with $B\le_{\rm m}C$.
If $B=C$, there is nothing to prove.
Assume $B\ne C$.
Since $e_B,e_C\in\PosId(\mathbf S)$, either $e_B\le e_C$ or $e_C\le e_B$.
If $e_C\le e_B$, then Lemma~\ref{lem:idempotents-max} gives
$e_Be_C=e_B$.
But $B\le_{\rm m}C$ implies $e_Be_C\in\layer{C}$; since this product is
a positive idempotent, it belongs to $C$.
Hence $e_B\in C$, contradicting $B\ne C$.
Thus $e_B\le e_C$.

The weaker linearity of the quotient skeleton does not suffice, as Example~\ref{ex:block-linear-not-skeleton-monotone-minima} shows.
\end{remark}

\begin{example}[$\tau$-block-linearity forces neither receiving-block skeleton monotonicity nor unitality of the collapse maps]
\label{ex:block-linear-not-skeleton-monotone-minima}
Let
\[
S=\{d,c,u,v,z\},
\]
ordered by the transitive closure of
\[
d<c,\qquad d<u,\qquad c<v,\qquad u<v,\qquad z<v,
\]
and define a commutative multiplication by
\[
\begin{array}{c|ccccc}
\cdot & d&c&u&v&z\\
\hline
d&d&c&u&v&z\\
c&c&c&v&v&v\\
u&u&v&u&v&z\\
v&v&v&v&v&v\\
z&z&v&z&v&v
\end{array}
\]

Associativity may be seen without checking all triples.
Let $\mathbf E=\langle\{0,1\},\wedge\rangle$, and let $\mathbf N=\langle\{0,a,1\},\cdot\rangle$ be the commutative monoid in which $1$ is the identity, $0$ is absorbing, and $a^2=0$.
The set
\[
\{(1,0),(0,1),(0,a),(0,0)\}
\]
is a subsemigroup of $\mathbf E\times\mathbf N$, and its multiplication becomes the restriction of the displayed table to $\{c,u,z,v\}$ under
\[
c=(1,0),\qquad u=(0,1),\qquad z=(0,a),\qquad v=(0,0).
\]
The element $d$ is then adjoined as an identity.
The five generating inequalities of the order are preserved by multiplication by each element, as follows directly from the table.
Hence $\mathbf S=\langle S,\le,\cdot\rangle$ is an ordered semigroup.

The positive idempotents are $d,c,u,v$, and the local-unit sets are
\[
\begin{aligned}
\{q:dq=d\}&=\{d\},&
\{q:cq=c\}&=\{d,c\},\\
\{q:uq=u\}&=\{d,u\},&
\{q:vq=v\}&=S,\\
\{q:zq=z\}&=\{d,u\}.
\end{aligned}
\]
Thus $\mathbf S$ is local-unit-aligned, with
\[
\tau(d)=d,\qquad
\tau(c)=c,\qquad
\tau(u)=u,\qquad
\tau(v)=v,\qquad
\tau(z)=u.
\]

Put
\[
D:=\{d\},\qquad A:=\{c\},\qquad B:=\{u,v\}.
\]
Then
\[
\layer{D}=\{d\},\qquad
\layer{A}=\{c\},\qquad
\layer{B}=\{u,v,z\},
\]
and
\[
\begin{aligned}
\tprod{D}{D}&=D,
&
\tprod{D}{A}&=A,
&
\tprod{D}{B}&=B,
\\
\tprod{A}{A}&=A,
&
\tprod{A}{B}&=\{v\}\subseteq B,
&
\tprod{B}{B}&=B.
\end{aligned}
\]
Hence $\{D,A,B\}$ is $\tau$-multiplication-coherent and its quotient skeleton is the chain
\[
D<_{\rm m}A<_{\rm m}B.
\]

This partition is canonical.
Indeed, let $\mathcal Q$ be any $\tau$-multiplication-coherent partition, and let $U$ be its block containing $u$.
Since $\tau(z)=u$, both $u$ and $z$ belong to $\layer{U}$.
But
\[
\tau(u^2)=u
\qquad\text{and}\qquad
\tau(z^2)=\tau(v)=v.
\]
Multiplication coherence therefore forces $u$ and $v$ into the same block of $\mathcal Q$.
Consequently,
\[
\Cm(\mathbf S)=\{D,A,B\}.
\]

The receiving set is $\mathcal R=\{A,B\}$.
Both receiving blocks are lower-pointed, with $e_A=c$ and $e_B=u$, but $c$ and $u$ are incomparable.
Thus
\[
A<_{\rm m}B
\qquad\text{while}\qquad
e_A\nleq e_B.
\]
Therefore $\tau$-block-linearity does not force the receiving-block skeleton monotonicity needed for the least-element collapse.

The same multiplication also shows that collapse maps need not be unital.
Consider the induced ordered subsemigroup
\[
\mathbf S':=\mathbf S\mathbin{\upharpoonright}\{c,u,v,z\}.
\]
Its canonical skeleton consists of $A=\{c\}<_{\rm m}B=\{u,v\}$, and $B$ is the only receiving block, so its least-element collapse is defined.
The two components are monoids with identities $c$ and $u$, respectively, but
\[
\rho_{A,B}(c)=cu=v\ne u.
\]
Hence the proper collapse map is an isotone semigroup homomorphism but not a unital monoid homomorphism.
\end{example}

\begin{proposition}\label{prop:least-element-collapse}
Assume that the least-element collapse of Definition~\ref{def:least-element-collapse} is defined.
Then, for all $A,B\in\Cm(\mathbf S)$ with $A<_{\rm m}B$,
\[
\rho_{A,B}=\lambda^{A,B}_{e_B}.
\]
Moreover, for every $q\in B$ and every $x\in\layer{A}$,
\[
\rho_{A,B}(x)\le\lambda^{A,B}_{q}(x).
\]
Each receiving component $\layer{B}$, with $A<_{\rm m}B$ for some $A$, is a monoid with identity $e_B$, and each proper map
\[
\rho_{A,B}\colon\layer{A}\to\layer{B}
\]
with $A<_{\rm m}B$ is an isotone semigroup homomorphism.
The maps $\rho_{A,B}$ satisfy the ordinary direct-system identities
\[
\rho_{A,A}=\mathrm{id}_{\layer{A}},
\qquad
\rho_{B,C}\circ\rho_{A,B}=\rho_{A,C}
\quad(A\le_{\rm m}B\le_{\rm m}C).
\]
\end{proposition}

\begin{proof}
If $B\in\mathcal R$, then Proposition~\ref{prop:lower-pointed-block-component-monoid} shows that $e_B$ is the identity of $\layer{B}$.
This proves the assertion about receiving components.

Assume $A<_{\rm m}B$.
Since $q\in B$ and $e_B$ is the least element of $B$, we have $e_B\le q$.
Isotonicity gives
$
xe_B\le xq=\lambda^{A,B}_{q}(x).
$
Thus $\rho_{A,B}=\lambda^{A,B}_{e_B}$ is the pointwise least member of the resolved-transport family.

The map $\rho_{A,B}$ lands in $\layer{B}$ by Proposition~\ref{prop:resolved-transport-basic}\eqref{item:ct-well-defined}.
It is isotone because $x\le y$ implies $xe_B\le ye_B$.
For $x,y\in\layer{A}$, centrality of $e_B$ gives
\[
\rho_{A,B}(x)\rho_{A,B}(y)
=(xe_B)(ye_B)
=xy e_B^2
=xy e_B
=\rho_{A,B}(xy),
\]
so $\rho_{A,B}$ is a semigroup homomorphism.

The identity $\rho_{A,A}=\mathrm{id}_{\layer{A}}$ holds by definition.
Let $A\le_{\rm m}B\le_{\rm m}C$.
If $A=B$ or $B=C$, the identity $\rho_{B,C}\circ\rho_{A,B}=\rho_{A,C}$ follows immediately from the diagonal definitions.
It remains to consider $A<_{\rm m}B<_{\rm m}C$.
Then $B,C\in\mathcal R$, so skeleton-monotonicity gives $e_B\le e_C$, and Lemma~\ref{lem:idempotents-max} gives $e_Be_C=e_C$.
Hence
\[
(\rho_{B,C}\circ\rho_{A,B})(x)=(xe_B)e_C=x(e_Be_C)=xe_C=\rho_{A,C}(x).
\]
Thus the least-element collapse is an ordinary direct system of ordered semigroups extracted from the resolved-transport system.
\end{proof}

Lower-pointedness of the receiving blocks supplies the least-target maps, while skeleton-monotonicity of their least idempotents supplies coherence under composition; together these hypotheses produce a direct system in the category of ordered semigroups.
A block outside $\mathcal R$ need not be lower-pointed, and hence its component need not be a monoid.
Moreover, even when the source and target components are monoids, a proper transition need not preserve their identities, as Example~\ref{ex:block-linear-not-skeleton-monotone-minima} shows.
Thus no unitality assertion is part of Proposition~\ref{prop:least-element-collapse}.

\begin{corollary}[Product recovery by the least-element collapse]
\label{cor:least-element-collapse-product}
Assume that the least-element collapse is defined.
Let $A,B\in\Cm(\mathbf S)$, put $C:=A\vee B$, and let $x\in\layer{A}$ and $y\in\layer{B}$.
Then
\[
xy=\rho_{A,C}(x)\rho_{B,C}(y),
\]
with the product computed in $\layer{C}$.
Consequently, the ordinary direct system of ordered semigroups given by the least-element collapse reconstructs the multiplication of $\mathbf S$.
\end{corollary}

\begin{proof}
If $A=B=C$, both transition maps are identities and the formula is immediate.
Otherwise $C\in\mathcal R$, so $e_C$ is the identity of $\layer{C}$.
Each proper map into $\layer{C}$ is multiplication by $e_C$, while a diagonal map is the identity.
Hence the product on the right is either $xy$ or $xye_C$.
Since $xy\in\layer{C}$, the latter also equals $xy$.
\end{proof}

Lower-pointedness therefore selects one resolved transport uniformly for each proper receiving component, and skeleton-monotonicity makes these selections coherent across successive transitions.
When a receiving block is not lower-pointed, no single transition homomorphism need work uniformly on its whole component, and the full target-indexed family may remain essential.

\begin{remark}\label{rem:no-lower-pointed-needed}
Lower-pointedness is not part of the canonical resolved-transport decomposition.
It enters only in the ordinary direct-system specialization: every receiving block must be lower-pointed, and the least idempotents of the receiving blocks must be skeleton-monotone.
Even then the resulting system is a direct system of ordered semigroups rather than necessarily a unital direct system of monoids; Example~\ref{ex:block-linear-not-skeleton-monotone-minima} gives a proper transition that does not preserve component identities.
The resolved transports reconstruct multiplication in arbitrary skeletons by Theorem~\ref{thm:resolved-product-recovery} and reconstruct lower-to-higher comparisons by Theorem~\ref{thm:resolved-order-recovery}.
Full order recovery additionally requires skeleton orientation, skeleton-supported cross-block totality, skeleton-supported strict upper-shadow reflection, directed lexicographicity, componentwise order duality, or another explicit cross-block order principle.
\end{remark}

We now turn from multiplication to the additional order rule naturally associated with the single-valued collapse maps.

\begin{definition}[Directed-lexicographic relation for the least-element collapse]\label{def:ordinary-directed-lexicographic-order}
This is the global relation obtained by applying the single-valued transported comparison rule to the least-element collapse maps, with arbitrary pairs compared after transport to their join component.
Assume that the least-element collapse $(\rho_{A,B})_{A\le_{\rm m}B}$ of Definition~\ref{def:least-element-collapse} is defined.

For $A,B\in\Cm(\mathbf S)$, $x\in\layer{A}$, and $y\in\layer{B}$, put $C:=A\vee B$.
Write $<_{C}$ for the strict part of the component order on $\layer{C}$.
The \emph{directed lexicographic relation} induced by the least-element collapse is the relation $\le_{\mathrm{dlex}}$ on $S$ defined by
\[
\begin{aligned}
x\le_{\mathrm{dlex}}y
\quad\Longleftrightarrow\quad{}
&\rho_{A,C}(x)<_{C}\rho_{B,C}(y),\\
&\text{or }\rho_{A,C}(x)=\rho_{B,C}(y)
  \text{ and }A\le_{\rm m}B.
\end{aligned}
\]
No assertion is made here that $\le_{\mathrm{dlex}}$ is a partial order for arbitrary least-element-collapse data.
We call it a directed lexicographic order only when it is a partial order.
We say that the ambient order is \emph{directed lexicographic with respect to the least-element collapse}, or simply that \emph{directed lexicographicity} holds, if the ambient order coincides with $\le_{\mathrm{dlex}}$.
\end{definition}

\begin{remark}[Working form of the directed lexicographic relation]\label{rem:ordinary-dlex-working-form}
With the notation of Definition~\ref{def:ordinary-directed-lexicographic-order}, let $x\in\layer{A}$, $y\in\layer{B}$, and $C=A\vee B$.
All inequalities on the right are evaluated in the indicated component.
Then
\[
x\le_{\mathrm{dlex}}y
\quad\Longleftrightarrow\quad
\begin{cases}
x\le y, & A=B,\\[1mm]
\rho_{A,B}(x)\le y, & A<_{\rm m}B,\\[1mm]
x<\rho_{B,A}(y), & B<_{\rm m}A,\\[1mm]
\rho_{A,C}(x)<_{C}\rho_{B,C}(y), & A\parallel B.
\end{cases}
\]
Here $A\parallel B$ means that $A$ and $B$ are incomparable.
Thus the lower-to-higher comparison is the usual transported comparison, the higher-to-lower comparison is strict below the transported lower element, and an incomparable-block comparison is made only after both elements have been transported to their join component.
For a chain skeleton this reduces to the finite directed lexicographic order, rewritten with the present paper's $\tau$-multiplication-coherent blocks as indices and the least-element collapse maps as transition homomorphisms.
\end{remark}

The directed lexicographic relation is not merely a formal candidate.
The direct-system structure gives natural sufficient conditions under which it
is an order, and in the totally ordered case it necessarily recovers the
ambient order.

\begin{proposition}[Orderhood criteria for the directed lexicographic relation]
\label{prop:dlex-orderhood-criteria}
Assume that the least-element collapse is defined.

\begin{enumerate}

\item\label{item:dlex-order-embedding-criterion}
If every proper transition
\[
\rho_{A,B}\colon\layer{A}\longrightarrow\layer{B},
\qquad
A<_{\rm m}B,
\]
is an order embedding, then $\le_{\mathrm{dlex}}$ is a partial order on
$S$.

\item\label{item:dlex-chain-criterion}
If $\mathbf S$ is $\tau$-block-linear and every component
$\layer{A}$ is a chain, then $\le_{\mathrm{dlex}}$ is a total order on
$S$.

\end{enumerate}
\end{proposition}

\begin{proof}
Reflexivity follows from
$\rho_{A,A}=\mathrm{id}_{\layer{A}}$.

For antisymmetry, let $x\in\layer{A}$, $y\in\layer{B}$, and put
$C=A\vee B$.
If $x\le_{\mathrm{dlex}}y$ and $y\le_{\mathrm{dlex}}x$, neither
transported image can be strictly below the other.
Hence
$\rho_{A,C}(x)=\rho_{B,C}(y)$, $A\le_{\rm m}B$, and
$B\le_{\rm m}A$.
Thus $A=B=C$, and therefore $x=y$.

For item~\ref{item:dlex-order-embedding-criterion}, suppose that
$x\le_{\mathrm{dlex}}y\le_{\mathrm{dlex}}z$, where
$x\in\layer{A}$, $y\in\layer{B}$, and $z\in\layer{D}$, and put
$H:=A\vee B\vee D$.
Using the direct-system identities and the order-embedding hypothesis, the
first comparison transports either to
$\rho_{A,H}(x)<\rho_{B,H}(y)$, or to equality together with
$A\le_{\rm m}B$.
Similarly, the second transports either to
$\rho_{B,H}(y)<\rho_{D,H}(z)$, or to equality together with
$B\le_{\rm m}D$.
Transitivity in $\layer{H}$ therefore gives either
$\rho_{A,H}(x)<\rho_{D,H}(z)$, or equality together with
$A\le_{\rm m}D$.

Now put $G:=A\vee D$.
Since
$\rho_{A,H}=\rho_{G,H}\circ\rho_{A,G}$ and
$\rho_{D,H}=\rho_{G,H}\circ\rho_{D,G}$, and since
$\rho_{G,H}$ is an order embedding, the preceding strict inequality or
equality reflects back to $\layer{G}$.
Hence $x\le_{\mathrm{dlex}}z$.
Thus $\le_{\mathrm{dlex}}$ is transitive and therefore a partial order.

For item~\ref{item:dlex-chain-criterion}, first observe that any two
elements are $\le_{\mathrm{dlex}}$-comparable.
Indeed, if $A<_{\rm m}B$, then the chain order of $\layer{B}$ gives
either $\rho_{A,B}(x)\le y$ or $y<\rho_{A,B}(x)$, yielding
respectively $x\le_{\mathrm{dlex}}y$ or
$y\le_{\mathrm{dlex}}x$.

It remains to prove transitivity.
If $A=B$ or $B=D$, the conclusion follows directly from transitivity
in the repeated component and isotonicity of the relevant transition map.
Suppose that $A=D\ne B$. If $B<_{\rm m}A$, then
$x<\rho_{B,A}(y)\le z$, so $x\le z$. If $A<_{\rm m}B$, then
$\rho_{A,B}(x)\le y<\rho_{A,B}(z)$. Were $z<x$, isotonicity would give
$\rho_{A,B}(z)\le\rho_{A,B}(x)$, a contradiction. Since
$\layer{A}$ is a chain, $x\le z$.
For three distinct blocks, the six possible relative positions of
$A,B,D$ are handled, using the direct-system identities and isotonicity,
by
\[
\begin{array}{rcl}
A<B<D
&:&
\rho_{A,D}(x)\le\rho_{B,D}(y)\le z,
\\[1mm]
A<D<B
&:&
\rho_{A,B}(x)\le y<\rho_{D,B}(z)
\ \Longrightarrow\
\rho_{A,D}(x)<z,
\\[1mm]
B<A<D
&:&
x<\rho_{B,A}(y)
\ \Longrightarrow\
\rho_{A,D}(x)\le\rho_{B,D}(y)\le z,
\\[1mm]
B<D<A
&:&
x<\rho_{B,A}(y)\le\rho_{D,A}(z),
\\[1mm]
D<A<B
&:&
\rho_{A,B}(x)\le y<\rho_{D,B}(z)
\ \Longrightarrow\
x<\rho_{D,A}(z),
\\[1mm]
D<B<A
&:&
x<\rho_{B,A}(y)\le\rho_{D,A}(z).
\end{array}
\]
In the second and fifth lines, the indicated conclusion follows because the
relevant source component is a chain and the corresponding transition is
isotone.
Each case yields $x\le_{\mathrm{dlex}}z$.
Thus $\le_{\mathrm{dlex}}$ is transitive and hence a total order.
\end{proof}

\begin{corollary}[Automatic directed lexicographicity in the totally ordered collapse case]
\label{cor:total-order-dlex-automatic}
Suppose that $\mathbf S$ is totally ordered and that its least-element
collapse is defined.
Then
\[
\le\;=\;\le_{\mathrm{dlex}}.
\]
Consequently, $\le_{\mathrm{dlex}}$ is a multiplication-compatible total
order and the least-element-collapse data reconstruct the full ordered
semigroup.
\end{corollary}

\begin{proof}
Totality of the ambient order implies that the canonical skeleton and every
component are chains.
By Proposition~\ref{prop:dlex-orderhood-criteria}
\eqref{item:dlex-chain-criterion}, the relation
$\le_{\mathrm{dlex}}$ is a total order.

It remains to compare it with the ambient order.
The two orders agree inside each component.
Let $A<_{\rm m}B$, $x\in\layer{A}$, and $y\in\layer{B}$.
Since $\rho_{A,B}=\lambda^{A,B}_{e_B}$ is the pointwise least
resolved transport, Theorem~\ref{thm:resolved-order-recovery} gives
\[
x\le y
\quad\Longleftrightarrow\quad
\rho_{A,B}(x)\le y.
\]
If this condition fails, the chain order of $\layer{B}$ gives
\[
y<\rho_{A,B}(x).
\]
Then $x\nleq y$, and ambient totality forces $y\le x$.
These are exactly the two comparable-block clauses of
Remark~\ref{rem:ordinary-dlex-working-form}.
Since the canonical skeleton is a chain, there are no incomparable-block
cases.
Hence the two orders coincide.
\end{proof}

The four cases above also locate directed lexicographicity relative to the
three order-recovery packages introduced earlier.
The comparable-block clauses control strict upper shadows, while the
incomparable-block clause determines whether skeleton-supportedness holds;
skeleton orientation and cross-block totality are then governed by the
position of the transition images inside the receiving components.

\begin{proposition}[Directed lexicographicity and the three order-recovery packages]
\label{prop:dlex-three-order-recovery-packages}
Assume that the least-element collapse is defined and that the ambient order is
directed lexicographic with respect to it.
Then the following statements hold.
\begin{enumerate}
\item\label{item:dlex-shadow-automatic}
The ambient order is strict upper-shadow reflective along the skeleton.

\item\label{item:dlex-skeleton-supported}
The ambient order is skeleton-supported if and only if, whenever
$A\parallel B$, $x\in\layer{A}$, $y\in\layer{B}$, and
$C=A\vee B$, the elements
\[
\rho_{A,C}(x)
\qquad\text{and}\qquad
\rho_{B,C}(y)
\]
are not strictly comparable in $\layer{C}$.

\item\label{item:dlex-skeleton-oriented}
The ambient order is skeleton-oriented if and only if the condition in
item~\ref{item:dlex-skeleton-supported} holds and, for every
$A<_{\rm m}B$ and $x\in\layer{A}$, the element
$
\rho_{A,B}(x)
$
is minimal in $\layer{B}$.

\item\label{item:dlex-cross-block-total}
The ambient order is cross-block total along the skeleton if and only if,
for every $A<_{\rm m}B$, $x\in\layer{A}$, and $y\in\layer{B}$, the
elements
$
\rho_{A,B}(x)
$
and
$
y
$
are comparable in $\layer{B}$.
\end{enumerate}

Consequently, a directed-lexicographic order satisfies the three
order-recovery packages considered above precisely as follows:
\begin{enumerate}
\item it satisfies the skeleton-orientation package exactly under the
conditions in item~\ref{item:dlex-skeleton-oriented};
\item it satisfies the skeleton-supported cross-block-totality package
exactly when the conditions in
items~\ref{item:dlex-skeleton-supported} and
\ref{item:dlex-cross-block-total} both hold;
\item it satisfies the skeleton-supported strict-upper-shadow package
exactly under the condition in
item~\ref{item:dlex-skeleton-supported}.
\end{enumerate}
\end{proposition}

\begin{proof}
Let $A<_{\rm m}B$, $x\in\layer{A}$, and $y\in\layer{B}$.
By Remark~\ref{rem:ordinary-dlex-working-form},
\[
y\le x
\quad\Longleftrightarrow\quad
y<_{B}\rho_{A,B}(x).
\]
Moreover, $\rho_{A,B}=\lambda^{A,B}_{e_B}$ is the pointwise least member
of the resolved-transport family.
Therefore
\[
y<_{B}\rho_{A,B}(x)
\quad\Longleftrightarrow\quad
y<_{B}\lambda^{A,B}_{q}(x)\text{ for every }q\in B.
\]
This proves item~\ref{item:dlex-shadow-automatic}.

Suppose that $A\parallel B$, and put $C=A\vee B$.
By the incomparable-block case of
Remark~\ref{rem:ordinary-dlex-working-form},
\[
x\le y
\quad\Longleftrightarrow\quad
\rho_{A,C}(x)<_{C}\rho_{B,C}(y),
\]
and symmetrically
\[
y\le x
\quad\Longleftrightarrow\quad
\rho_{B,C}(y)<_{C}\rho_{A,C}(x).
\]
Thus there is no comparison between elements from incomparable skeleton
blocks exactly when their two join-component shadows are not strictly
comparable.
This proves item~\ref{item:dlex-skeleton-supported}.

For $A<_{\rm m}B$, the only comparisons forbidden by skeleton orientation
are reverse comparisons from $\layer{B}$ to $\layer{A}$.
Under directed lexicographicity,
\[
y\le x
\quad\Longleftrightarrow\quad
y<_{B}\rho_{A,B}(x).
\]
Such reverse comparisons are absent for every $y\in\layer{B}$ exactly when
$\rho_{A,B}(x)$ is minimal in $\layer{B}$.
Together with item~\ref{item:dlex-skeleton-supported}, this proves
item~\ref{item:dlex-skeleton-oriented}.

Finally, for $A<_{\rm m}B$,
\[
x\le y
\quad\Longleftrightarrow\quad
\rho_{A,B}(x)\le y,
\qquad
y\le x
\quad\Longleftrightarrow\quad
y<_{B}\rho_{A,B}(x).
\]
Hence $x$ and $y$ are comparable exactly when
$\rho_{A,B}(x)$ and $y$ are comparable inside $\layer{B}$.
This proves item~\ref{item:dlex-cross-block-total}.
The final three assertions now follow from
items~\ref{item:dlex-shadow-automatic}--\ref{item:dlex-cross-block-total}.
\end{proof}

\begin{corollary}[$\tau$-block-linear directed-lexicographic orders]
\label{cor:block-linear-dlex-order-packages}
Under the hypotheses of
Proposition~\ref{prop:dlex-three-order-recovery-packages}, suppose that
$\mathbf S$ is $\tau$-block-linear.
Then the ambient order is automatically skeleton-supported and strict
upper-shadow reflective.
Moreover:
\begin{enumerate}
\item it is skeleton-oriented if and only if every proper transition image
$\rho_{A,B}(x)$ is minimal in its receiving component;
\item it is cross-block total along the skeleton if and only if every proper
transition image is comparable with every element of its receiving component.
\end{enumerate}
In particular, if every receiving component is a chain, then the ambient order
is cross-block total along the skeleton.
\end{corollary}

\begin{proof}
There are no incomparable canonical skeleton blocks, so the condition in
Proposition~\ref{prop:dlex-three-order-recovery-packages}
\eqref{item:dlex-skeleton-supported} is vacuous.
The remaining assertions follow from
items~\eqref{item:dlex-skeleton-oriented} and
\eqref{item:dlex-cross-block-total}.
\end{proof}

\begin{remark}[Order recovery under the directed lexicographic rule]
\label{rem:dlex-collapse-full-order-recovery}
In particular, whenever directed lexicographicity holds, the component orders,
the canonical skeleton, and the least-element collapse maps determine every
ambient comparison through the four cases displayed in
Remark~\ref{rem:ordinary-dlex-working-form}.
Proposition~\ref{prop:dlex-orderhood-criteria} gives intrinsic conditions
under which the directed lexicographic relation is a partial or total order,
and Corollary~\ref{cor:total-order-dlex-automatic} shows that it necessarily
coincides with the ambient order in the totally ordered collapse case.
Outside such situations, coincidence with the ambient order remains an
additional order-side condition; no claim is made that arbitrary
least-element-collapse data make $\le_{\mathrm{dlex}}$ an ordered-semigroup
order.
\end{remark}

The remaining implications among the least-element-collapse order principles are as follows.

\begin{proposition}[Implications among the least-element-collapse order principles]
\label{prop:least-element-collapse-order-implications}
Assume that the least-element collapse is defined.
Then the following statements hold.
\begin{enumerate}
\item\label{item:cross-total-implies-shadow}
Cross-block totality along the skeleton implies strict upper-shadow reflection along the skeleton.

\item\label{item:shadow-implies-cross-total}
Suppose that every receiving component is a chain, in the sense that $\layer{B}$ is a chain whenever $A<_{\rm m}B$ for some skeleton block $A$.
Then strict upper-shadow reflection along the skeleton implies cross-block totality along the skeleton.

\item\label{item:chain-skeleton-shadow-dlex}
If $\mathbf S$ is $\tau$-block-linear, then strict upper-shadow
reflection is equivalent to directed lexicographicity.
If, in addition, every receiving component is a chain, then these two properties are also equivalent to cross-block totality along the skeleton.

\item\label{item:totality-characterization}
The ambient order is total if and only if $\mathbf S$ is
$\tau$-block-linear, every component is a chain, and the order is
cross-block total along the canonical skeleton.
Under the two chain conditions, cross-block totality may equivalently be replaced by strict upper-shadow reflection or by directed lexicographicity.
\end{enumerate}
\end{proposition}

\begin{proof}
For comparable blocks $A<_{\rm m}B$, the least-element-collapse form of Theorem~\ref{thm:resolved-order-recovery} is
\[
x\le y
\quad\Longleftrightarrow\quad
\rho_{A,B}(x)\le y
\qquad(x\in\layer{A},\ y\in\layer{B}).
\]
Moreover, because $\rho_{A,B}=\lambda^{A,B}_{e_B}$ is the pointwise least member of the resolved-transport family,
\[
y<_{B}\lambda^{A,B}_{q}(x)\text{ for every }q\in B
\quad\Longleftrightarrow\quad
y<_{B}\rho_{A,B}(x).
\]

For item~\ref{item:cross-total-implies-shadow}, let $A<_{\rm m}B$, $x\in\layer{A}$, and $y\in\layer{B}$.
If $y\le x$, isotonicity after multiplication by $e_B$ gives
\[
y=ye_B\le xe_B=\rho_{A,B}(x).
\]
Equality would imply $\rho_{A,B}(x)\le y$, hence $x\le y$, contradicting antisymmetry because the two components are disjoint.
Therefore $y<_{B}\rho_{A,B}(x)$.
Conversely, if $y<_{B}\rho_{A,B}(x)$, then the lower-to-higher test shows that $x\nleq y$.
Cross-block totality now forces $y\le x$.
Hence strict upper-shadow reflection holds.

For item~\ref{item:shadow-implies-cross-total}, let $A<_{\rm m}B$, $x\in\layer{A}$, and $y\in\layer{B}$.
Since the receiving component $\layer{B}$ is a chain, either $\rho_{A,B}(x)\le y$ or $y<_{B}\rho_{A,B}(x)$.
The lower-to-higher test gives $x\le y$ in the first case, while strict upper-shadow reflection gives $y\le x$ in the second.
Thus every pair from comparable distinct blocks is comparable.

Suppose that $\mathbf S$ is $\tau$-block-linear.
Directed lexicographicity implies strict upper-shadow reflection by
Proposition~\ref{prop:dlex-three-order-recovery-packages}
\eqref{item:dlex-shadow-automatic}.
Conversely, assume strict upper-shadow reflection.
There are no incomparable canonical skeleton blocks.
Inside one component and from a lower block to a higher block, the ambient
order agrees with the corresponding clauses of
Remark~\ref{rem:ordinary-dlex-working-form}; along comparable blocks, strict
upper-shadow reflection gives the remaining higher-to-lower clause.
Hence the ambient order is directed lexicographic.
This proves the first assertion of item~\ref{item:chain-skeleton-shadow-dlex};
its second assertion follows from
items~\ref{item:cross-total-implies-shadow} and
\ref{item:shadow-implies-cross-total}.

If the ambient order is total, every component is a chain, cross-block
totality is automatic, and $\mathbf S$ is $\tau$-block-linear by
Proposition~\ref{prop:m-skeleton-join}.
Conversely, assume that $\mathbf S$ is $\tau$-block-linear, that all
components are chains, and that cross-block totality holds.
Two elements in the same component are comparable by the component-chain hypothesis, while two elements in distinct components lie in comparable skeleton blocks and are comparable by cross-block totality.
Hence the ambient order is total.
The final equivalences in item~\ref{item:totality-characterization} follow from item~\ref{item:chain-skeleton-shadow-dlex}.
\end{proof}

The unit-constant case further links the absence of reverse comparisons to endpoint identities in the receiving components.

\begin{remark}[Skeleton-orientation and endpoint units in the unit-constant collapse]
\label{rem:orientation-and-least-units}
Assume that $A<_{\rm m}B$, that the least-element-collapse transition from $\layer{A}$ to $\layer{B}$ is unit-constant, and that $e_B$ is the identity of $\layer{B}$:
\[
\rho_{A,B}(x)=e_B
\qquad(x\in\layer{A}).
\]
Under the directed-lexicographic comparison rule, an element $y\in\layer{B}$ lies below $x\in\layer{A}$ exactly when
$
y<\rho_{A,B}(x)=e_B.
$
Consequently, there is no reverse comparison from $\layer{B}$ to $\layer{A}$ exactly when $e_B$ is minimal in $\layer{B}$.
If $\layer{B}$ is a chain, minimality is equivalent to $e_B$ being the least element of $\layer{B}$.

Thus, when $\mathbf S$ is $\tau$-block-linear and all receiving
components are chains, skeleton orientation in the unit-constant
directed-lexicographic setting is equivalent to the identity of every
receiving component being its least element.
Without the assumption of $\tau$-block-linearity, the same conclusion
additionally requires the exclusion of comparisons between incomparable
canonical skeleton blocks, for example by skeleton-supportedness.
\end{remark}

The resulting mixed-product rule has a classical analogue.

\begin{remark}[Comparison with Clifford ordinal sums]
\label{rem:clifford-ordinal-sums-unit-constant-case}
The unit-constant least-element collapse produces the same mixed-product absorption law that appears in Clifford's ordinal-sum theory for naturally totally ordered commutative semigroups \cite{Clifford1954,Clifford1958}.
Indeed, suppose that $A<_{\rm m}B$, $x\in\layer{A}$, and $y\in\layer{B}$, and assume that
\[
\rho_{A,B}(x)=e_B.
\]
Then
\[
xy=\rho_{A,B}(x)y=e_By=y.
\]
Thus the higher component absorbs the lower component in mixed products.

This agreement of the mixed-product rule does not by itself make the resulting decomposition a Clifford ordinal sum.
Such an identification additionally requires the relevant classical hypotheses, including commutativity, an appropriate chain structure, naturality or totality of the order, and compatible endpoint conventions.
Accordingly, unit-constant least-element collapse should be viewed here as a Clifford-type absorption special case of the resolved-transport mechanism.
\end{remark}

The examples in Appendix~\ref{sec:dependence-order-recovery-principles} show that strict upper-shadow reflection need not imply cross-block totality without the hypothesis that every receiving component is a chain, and that skeleton orientation is independent of the chain-skeleton order principles.

The preceding sections prove the reconstruction ingredients separately.
At the $\Cm$-level, the quotient skeleton and resolved transports recover multiplication and all lower-to-higher comparisons.
The order-recovery packages of Section~\ref{sec:canonical-resolved-transport} specify when the same data recover the remaining comparisons.
The theorem below collects these facts into the main reconstruction statement for the present paper.

\section{The canonical resolved-transport reconstruction theorem}
\label{sect:canonical-resolved-decomposition}

\begin{theorem}[Canonical resolved-transport decomposition and reconstruction]
\label{thm:canonical-resolved-decomposition}
Let $\mathbf S$ be a local-unit-aligned ordered semigroup.
Then $\mathbf S$ has the $\tau$-cohesive decomposition $\Ltau(\mathbf S)$, induced by the finest $\tau$-stable partition $\Ctau(\mathbf S)$, and the canonical $\tau$-multiplication-coherent decomposition $\Lm(\mathbf S)$, induced by the finest $\tau$-multiplication-coherent partition $\Cm(\mathbf S)$.

The partition $\Cm(\mathbf S)$ is the multiplicative coherentization threshold: it is the finest $\tau$-multiplication-coherent partition of $\PosId(\mathbf S)$, equivalently, the finest partition for which the block containing $\tau(xy)$ is determined by the blocks containing $\tau(x)$ and $\tau(y)$.
Its blocks form a join-semilattice and index the canonical resolved-transport system $\mathfrak R(\mathbf S)$.
For $A,B\in\Cm(\mathbf S)$, put $C:=A\vee B$.
For $x\in\layer{A}$ and $y\in\layer{B}$, the product is recovered by
\[
xy=
\min\{\lambda^{A,C}_{q}(x)\lambda^{B,C}_{q}(y):q\in C\},
\]
computed inside $\layer{C}$.
Thus $\mathfrak R(\mathbf S)$ reconstructs the ambient multiplication.

The same data recover every lower-to-higher comparison.
If $A\le_{\rm m}B$, $x\in\layer{A}$, and $y\in\layer{B}$, then
\[
x\le y
\quad\Longleftrightarrow\quad
\lambda^{A,B}_{q}(x)\le y\text{ for some }q\in B.
\]
No additional order hypothesis is needed for this part of the reconstruction.
\end{theorem}

\begin{proof}
The existence and canonical nature of $\Ltau(\mathbf S)$ and $\Lm(\mathbf S)$ are Theorems~\ref{thm:finest-tau-stable} and~\ref{thm:finest-coherent}.
For $\Cm(\mathbf S)$, Lemma~\ref{lem:m-block-product} and Proposition~\ref{prop:m-skeleton-join} give the block join, while Definition~\ref{def:canonical-resolved-transport-homomorphisms} gives $\mathfrak R(\mathbf S)$.
The product formula is Theorem~\ref{thm:resolved-product-recovery}, and the lower-to-higher comparison test is Theorem~\ref{thm:resolved-order-recovery}.
\end{proof}

\begin{corollary}[Full order recovery]
\label{cor:canonical-full-order-recovery}
In addition to the hypotheses of Theorem~\ref{thm:canonical-resolved-decomposition}, assume that $\mathbf S$ satisfies one of the following conditions:
\begin{enumerate}
\item $\mathbf S$ is skeleton-oriented;
\item $\mathbf S$ is skeleton-supported and cross-block total along the skeleton;
\item $\mathbf S$ is skeleton-supported and strict upper-shadow reflective along the skeleton.
\end{enumerate}
Then $\mathfrak R(\mathbf S)$ reconstructs the full ordered semigroup $\langle S,\le,\cdot\rangle$.
\end{corollary}

\begin{proof}
Apply Corollary~\ref{cor:oriented-full-order-recovery}, Proposition~\ref{prop:cross-block-totality-recovery}, or Proposition~\ref{prop:strict-upper-shadow-recovery}, respectively, together with the multiplication recovery in Theorem~\ref{thm:canonical-resolved-decomposition}.
\end{proof}

\begin{corollary}[Canonical recovery from componentwise order duality]
\label{cor:canonical-order-duality-recovery}
Suppose that $\mathbf S$ is skeleton-supported and carries a component-preserving order anti-automorphism $\nu$.
Then $\mathfrak R_{\nu}(\mathbf S)$ reconstructs the full expanded ordered semigroup $\langle S,\le,\cdot,\nu\rangle$.
If $\mathbf S$ is $\tau$-block-linear, skeleton-supportedness is automatic.
\end{corollary}

\begin{proof}
This is Proposition~\ref{prop:antiautomorphism-full-order-recovery} and Corollary~\ref{cor:antiautomorphism-chain-skeleton-recovery}, together with Theorem~\ref{thm:canonical-resolved-decomposition}.
\end{proof}

The specialization in which all proper receiving blocks are lower-pointed and their least idempotents are skeleton-monotone has already been separated from these unconditional reconstruction results.
Proposition~\ref{prop:least-element-collapse} extracts an ordinary direct system of ordered semigroups under these hypotheses, Corollary~\ref{cor:least-element-collapse-product} gives its product rule, and directed lexicographicity supplies the additional order-recovery condition of Remark~\ref{rem:dlex-collapse-full-order-recovery}.

\begin{corollary}[Totally ordered case]\label{cor:linearly-ordered-case}
If $\mathbf S$ is a totally ordered local-unit-aligned semigroup, then the canonical resolved-transport system $\mathfrak R(\mathbf S)$ reconstructs the full ordered semigroup $\langle S,\le,\cdot\rangle$.
\end{corollary}

\begin{proof}
If $S$ is totally ordered, then $\PosId(\mathbf S)$ is totally ordered.
By Proposition~\ref{prop:m-skeleton-join}, $\mathbf S$ is
$\tau$-block-linear.
Thus its canonical skeleton has no incomparable blocks, so
skeleton-supportedness is automatic.
Moreover, any two elements in distinct blocks are comparable in the ambient order, so the order is cross-block total along the skeleton.
The second condition in Corollary~\ref{cor:canonical-full-order-recovery} applies.
\end{proof}

Say that the canonical resolved-transport systems
$\mathfrak R(\mathbf S)$ and $\mathfrak R(\mathbf N)$ are
\emph{isomorphic} if there exist a join-semilattice isomorphism
$\alpha\colon\Cm(\mathbf S)\to\Cm(\mathbf N)$ and, for each
$A\in\Cm(\mathbf S)$, an ordered-semigroup isomorphism
$\phi_A\colon\layer{A}\to\{z\in N:\tau_{\mathbf N}(z)\in\alpha(A)\}$
such that $\phi_A[A]=\alpha(A)$,
$\phi_A(\tau_{\mathbf S}(x))=\tau_{\mathbf N}(\phi_A(x))$ for every
$x\in\layer{A}$, and, whenever $A\le_{\rm m}B$, $q\in B$, and
$x\in\layer{A}$,
\[
\phi_B\bigl(\lambda^{\mathbf S,A,B}_{q}(x)\bigr)
=
\lambda^{\mathbf N,\alpha(A),\alpha(B)}_{\phi_B(q)}
\bigl(\phi_A(x)\bigr).
\]

\begin{corollary}[Completeness of the canonical data]
\label{cor:canonical-data-complete-invariant}
Let $\mathbf S$ and $\mathbf N$ be local-unit-aligned ordered semigroups
such that, for some $i\in\{1,2,3\}$, both semigroups satisfy
condition~$i$ of Corollary~\ref{cor:canonical-full-order-recovery}.
Then
\[
\mathbf S\cong\mathbf N
\quad\Longleftrightarrow\quad
\mathfrak R(\mathbf S)\cong\mathfrak R(\mathbf N).
\]
Thus, within each of the three order-recovery classes, the canonical
resolved-transport system is a complete invariant up to ordered-semigroup
isomorphism.
\end{corollary}

\begin{proof}
Suppose first that
$\mathfrak R(\mathbf S)\cong\mathfrak R(\mathbf N)$, and define
$\Phi\colon S\to N$ by $\Phi(x)=\phi_A(x)$ for the unique
$A\in\Cm(\mathbf S)$ such that $x\in\layer{A}$.
The maps $\alpha$ and $\phi_A$ preserve the canonical skeleton,
component orders, component multiplications, and resolved transports.
The reconstruction formulas for multiplication and for the relevant
order-recovery package therefore show that $\Phi$ is an
ordered-semigroup isomorphism.

Conversely, an ordered-semigroup isomorphism
$\mathbf S\cong\mathbf N$ preserves the local-unit map, positive
idempotents, the canonical partition $\Cm$, its join operation, the
component ordered semigroups, and the resolved transports.
Its restrictions to the canonical components therefore induce an
isomorphism
$\mathfrak R(\mathbf S)\cong\mathfrak R(\mathbf N)$.
\end{proof}

\begin{remark}\label{rem:rigidity-omitted}
No rigidity, uniqueness of block units, or unit-constancy statement is included in the canonical reconstruction results.
Those conclusions belong to the finite totally ordered lower-pointed situation, in which receiving-minimum monotonicity is automatic, and should not be imported into the present infinite semigroup setting without additional hypotheses.
\end{remark}

\section{Conclusion}

The finest $\tau$-multiplication-coherent partition turns the positive-idempotent skeleton of a local-unit-aligned ordered semigroup into a join-semilattice of canonical components.
The associated resolved transports reconstruct multiplication and every lower-to-higher comparison.
The remaining order is recovered under any of the three explicit cross-component conditions isolated in the paper, and also from a component-preserving order anti-automorphism under skeleton-supportedness.
Balanced involutive residuated expansions provide one specialization of this order-duality mechanism.

Let $\mathcal R$ be the set of blocks receiving proper transitions.
If every block in $\mathcal R$ is lower-pointed and the receiving least idempotents are skeleton-monotone, the resolved system collapses to an ordinary direct system of ordered semigroups.
Its maps reconstruct multiplication, while full order recovery still requires an explicit order-side rule.
In this collapse setting, directed lexicographicity is characterized by
strict upper-shadow reflection together with the join-component comparison
for incomparable blocks; for $\tau$-block-linear semigroups it is
equivalent to strict upper-shadow reflection, and, when every receiving component is a chain, also to cross-block totality.
Totality together with skeleton orientation forces every receiving-component identity to be least, whereas skeleton orientation alone does not.

The examples distinguish the relevant sources of rigidity.
They show that $\tau$-block-linearity does not force receiving-block monotonicity, that resolved-transport data do not determine the missing reverse comparisons, and that the endpoint consequence of skeleton orientation depends essentially on totality.
The appendix extends the independence of skeleton orientation from the
chain-skeleton order principles, and the same-data order ambiguity, to
arbitrary nontrivial chain skeletons with a least element.
It also shows that the receiving-component chain hypothesis in the converse
from strict upper-shadow reflection to cross-block totality cannot be omitted.

The reconstruction obtained here is internal: it concerns the canonical data extracted from a given ordered semigroup. A natural next step is a component-level structure theory relating the two coherentizations. The finer decomposition $\Ltau(\mathbf S)$ records the canonical intrinsically $\tau$-cohesive pieces within the components of $\Lm(\mathbf S)$, whereas $\Lm(\mathbf S)$ records the coarser assemblies on which resolved-transport reconstruction becomes possible. This leads to the problems of characterizing $\tau$-cohesive local-unit-aligned ordered semigroups and determining the admissible multiplication patterns by which such components combine into $\tau$-multiplication-coherent components. A complementary direction is a converse realization theory for abstract resolved-transport systems.

\appendix

\section{Separation models for the order-recovery principles}
\label{sec:dependence-order-recovery-principles}

We use a common algebraic template over an arbitrary nontrivial chain
skeleton with a least element.
The transition type and ambient order are varied; the fifth row uses
non-chain component orders to test the receiving-component chain hypothesis.
The five rows show that the receiving-component chain hypothesis in
Proposition~\ref{prop:least-element-collapse-order-implications}
\eqref{item:shadow-implies-cross-total} cannot be omitted, that skeleton
orientation is independent of the equivalent chain-skeleton principles, and
that the same-data order ambiguity persists over arbitrary nontrivial chain skeletons with a least element.

Let $I$ be a nontrivial chain with least element $0$, let
$P=\mathbb N^2$, and write $0_P=(0,0)$ and $w=(0,1)$. We use both the
lexicographic order $\le_\ell$ and the coordinatewise order $\preceq$ on
$P$; let $\prec$ be the strict part of $\preceq$. Put
$S_I=I\times P$ and $e_i=(i,0_P)$. For $i\le j$ and
$\varepsilon\in\{0,1\}$, define
\[
\eta_\varepsilon^{i,j}(u):=
\begin{cases}
u, & i=j\text{ or }\varepsilon=1,\\
0_P, & i<j\text{ and }\varepsilon=0,
\end{cases}
\qquad
f_\varepsilon^{i,j}(i,u):=(j,\eta_\varepsilon^{i,j}(u)).
\]
For $k=\max\{i,j\}$, set
\[
(i,u)\cdot_\varepsilon(j,v)
:=
\bigl(k,\eta_\varepsilon^{i,k}(u)
       +\eta_\varepsilon^{j,k}(v)\bigr).
\]
Thus the proper maps $f_0^{i,j}$ are zero maps, whereas the proper maps
$f_1^{i,j}$ preserve the $P$-coordinate.

For $i,j\in I$, put $\delta_{ij}=0$ if $i\le j$, and
$\delta_{ij}=1$ if $i>j$. Define five orders on $S_I$ by
\[
\begin{aligned}
(i,u)\le_{\mathrm{dlex_{uc}}}(j,v)
&\iff i<j\ \text{or}\ (i=j\text{ and }u\le_\ell v),\\
(i,u)\le_\uparrow(j,v)
&\iff i\le j\text{ and }u\le_\ell v,\\
(i,u)\le_t(j,v)
&\iff u+t\delta_{ij}w\le_\ell v
\qquad(t=1,2),\\
(i,u)\le_{\mathrm{sh}}(j,v)
&\iff
\bigl(i\le j\text{ and }u\preceq v\bigr)
\ \text{or}\
\bigl(i>j\text{ and }u\prec v\bigr).
\end{aligned}
\]
Write $\le_{\mathrm{dlex_{\uparrow}}}:=\le_1$ and
$\le_{\mathrm{gap}}:=\le_2$, and consider
\begin{enumerate}
\taggeditem{(a)}{item:nontrivial-skeleton-oriented-ordinal}
$\mathbf S_{\mathrm{dlex_{uc}}}^{I}
 :=\langle S_I,\le_{\mathrm{dlex_{uc}}},\cdot_0\rangle$;
\taggeditem{(b)}{ex:non-unit-constant-collapse}
$\mathbf S_{\uparrow}^{I}
 :=\langle S_I,\le_\uparrow,\cdot_1\rangle$;
\taggeditem{(c)}{ex:same-direct-system-different-orders}
$\mathbf S_{\mathrm{dlex_{\uparrow}}}^{I}
 :=\langle S_I,\le_{\mathrm{dlex_{\uparrow}}},\cdot_1\rangle$;
\taggeditem{(d)}{ex:gap-reverse-not-shadow}
$\mathbf S_{\mathrm{gap}}^{I}
 :=\langle S_I,\le_{\mathrm{gap}},\cdot_1\rangle$;
\taggeditem{(e)}{ex:strict-shadow-not-cross-total}
$\mathbf S_{\mathrm{sh}}^{I}
 :=\langle S_I,\le_{\mathrm{sh}},\cdot_1\rangle$.
\end{enumerate}

\begin{lemma}[Common algebraic and collapse data]
\label{lem:common-separation-construction}
The five structures above are commutative local-unit-aligned ordered monoids.
In each of them,
\[
\PosId(\mathbf S)=\{e_i:i\in I\},
\qquad
\tau(i,u)=e_i,
\qquad
\Cm(\mathbf S)=\bigl\{\{e_i\}:i\in I\bigr\}.
\]
Hence the canonical skeleton is canonically isomorphic to $I$, every component is infinite, and the
proper least-element-collapse maps are $f_0^{i,j}$ in row~(a) and
$f_1^{i,j}$ in rows~(b)--(e).
\end{lemma}

\begin{proof}
The maps $\eta_\varepsilon^{i,j}$ are additive and compose along $I$.
Consequently, a finite product has maximal first coordinate $m$ and second
coordinate
\[
\sum_r\eta_\varepsilon^{i_r,m}(u_r),
\]
independently of its bracketing. Thus $\cdot_\varepsilon$ is associative
and commutative, $e_0$ is its identity, and its idempotents are precisely
the $e_i$.

The first two displayed orders are respectively an ordinal-sum order and a
product order. For $t=1,2$, transitivity of $\le_t$ follows from
$\delta_{ik}\le\delta_{ij}+\delta_{jk}$, and antisymmetry follows because
opposite cross-level comparisons would add the positive vector $tw$.
Transitivity of $\le_{\mathrm{sh}}$ follows because a composite
higher-to-lower comparison contains a strict coordinatewise step;
antisymmetry is immediate. For rows~(b)--(e), multiplication by $(k,z)$
sends
\[
(i,u)\longmapsto(\max\{i,k\},u+z).
\]
This preserves the first-coordinate requirements and the component orders;
for $\le_t$, use
$\delta_{\max\{i,k\},\max\{j,k\}}\le\delta_{ij}$.
For row~(a), isotonicity follows directly from the two ordinal-sum cases:
same-level comparisons are preserved by translation, while a strict
lower-level comparison remains so unless both products coincide or enter one
level, where $z\le_\ell v+z$. Hence all five structures are ordered
monoids.

In rows~(b)--(e), multiplication by $e_i$ either fixes $(j,u)$ or moves it
up to $(i,u)$; in row~(a), it either fixes $(j,u)$ or moves it to the
higher element $e_i$. Thus every $e_i$ is positive. In rows~(b)--(e), an
element $(j,v)$ is a local unit of $(i,u)$ exactly when $j\le i$ and
$v=0_P$; their greatest member is $e_i$. In row~(a), every lower-level
element is a local unit of
$(i,u)$, at level $i$ only $e_i$ is one, and all lower-level elements
lie below $e_i$. Thus $\tau(i,u)=e_i$ in every row. Products from levels
$i$ and $j$ lie in level $\max\{i,j\}$, so the singleton partition of
the $e_i$'s is $\tau$-multiplication-coherent and therefore equals
$\Cm(\mathbf S)$. Finally, multiplication by $e_j$ is precisely
$f_\varepsilon^{i,j}$, proving the collapse assertion.
\end{proof}

\begin{example}[Five separation families over arbitrary chain skeletons with a least element]
\label{ex:nontrivial-skeleton-oriented}
The order properties are as follows. Since $I$ is a chain, the
dlex/shadow column may equivalently be read as directed lexicographicity or
strict upper-shadow reflection, by
Proposition~\ref{prop:least-element-collapse-order-implications}
\eqref{item:chain-skeleton-shadow-dlex}.
\begin{center}
\scriptsize
\begin{tabular}{@{}c c c c c c c@{}}
\hline
row & collapse & comp. chain & total & skel.-oriented & dlex/shadow & cross-block total\\
\hline
(a) & $f_0^{i,j}$ & yes & yes & yes & yes & yes\\
(b) & $f_1^{i,j}$ & yes & no  & yes & no  & no \\
(c) & $f_1^{i,j}$ & yes & yes & no  & yes & yes\\
(d) & $f_1^{i,j}$ & yes & no  & no  & no  & no \\
(e) & $f_1^{i,j}$ & no  & no  & no  & yes & no \\
\hline
\end{tabular}
\end{center}

To verify the row-dependent entries, fix $r<s$ in $I$. The first four
orders satisfy
\[
\le_\uparrow\ \subsetneq\ \le_{\mathrm{gap}}\ \subsetneq\
\ \le_{\mathrm{dlex_{\uparrow}}},
\qquad
\le_\uparrow\ \subsetneq\ \le_{\mathrm{dlex_{uc}}},
\]
while $\le_{\mathrm{dlex_{uc}}}$ is incomparable with
$\le_{\mathrm{gap}}$ and $\le_{\mathrm{dlex_{\uparrow}}}$. The strict
inclusions and incomparabilities are witnessed by
\[
(s,0_P)\le_{\mathrm{gap}}(r,2w),
\qquad
(s,0_P)\le_{\mathrm{dlex_{\uparrow}}}(r,w),
\qquad
(r,w)\le_{\mathrm{dlex_{uc}}}(s,0_P).
\]
Rows~(a) and~(b) admit no higher-to-lower comparison and are
skeleton-oriented. The first two displayed witnesses show that rows~(c)
and~(d) are not, and
\[
(s,0_P)\le_{\mathrm{sh}}(r,(1,0))
\]
does the same for row~(e).

Row~(a) is the directed lexicographic order for the zero transitions. Since
\[
u<_\ell v\quad\Longleftrightarrow\quad u+w\le_\ell v,
\]
row~(c) is the directed lexicographic order for the identity transitions;
row~(e) has exactly the corresponding non-strict and strict clauses for
$\preceq$. Rows~(b) and~(d) fail strict upper-shadow reflection, because
\[
(s,0_P)<(s,w)=f_1^{r,s}(r,w),
\]
but $(s,0_P)$ is below $(r,w)$ in neither of their ambient orders.

Rows~(a) and~(c) are total. Cross-block totality fails in rows~(b) and~(d)
because $(r,w)$ and $(s,0_P)$ are incomparable, and it fails in row~(e)
because $(r,(1,0))$ and $(s,(0,1))$ are incomparable. This proves every
entry in the table.

Thus rows~(a)--(d) realize all four combinations of skeleton orientation and
the equivalent chain-component principles. Rows~(b)--(d) have identical
component ordered monoids, multiplication, skeleton, and collapse maps but
different ambient orders. Row~(e) shows that the receiving-component chain
hypothesis in
Proposition~\ref{prop:least-element-collapse-order-implications}
\eqref{item:shadow-implies-cross-total} cannot be omitted. Since $I$ was
arbitrary, all these separations hold over chain skeletons of unrestricted
cardinality and order type, subject only to the presence of a least element.
\end{example}

\section*{Declarations}

\noindent\textbf{Ethical approval.}
Not applicable.

\noindent\textbf{Competing interests.}
The author declares that he has no competing interests.

\noindent\textbf{Authors' contribution.}
The author was solely responsible for the conception, formal analysis, investigation, and writing of the manuscript.

\noindent\textbf{Availability of data and materials.}
No datasets were generated or analyzed during the current study.

\noindent\textbf{Funding.}
This work was supported by the Ministry of Culture and Innovation of Hungary, through the National Research, Development and Innovation Fund, grant no.~K138596.

\end{document}